\documentclass[hidelinks,onefignum,onetabnum]{siamart220329}

\usepackage{algorithm}
\usepackage{algpseudocode}
\usepackage{amsfonts}
\usepackage{amsmath}
\usepackage{booktabs} 
\usepackage{caption}
\usepackage{capt-of}
\usepackage{comment}
\usepackage{epstopdf}
\usepackage{enumitem}   
\usepackage{float}
\graphicspath{ {./img/} }
\usepackage{multirow}
\usepackage{nicefrac}
\usepackage{longtable}
\usepackage{mathrsfs}
\usepackage{subcaption}
\usepackage{soul}
\usepackage{wrapfig}
\usepackage{hyperref}
\usepackage{cleveref}

\makeatletter
\def\cref@override@label@type#1\@nil#2{%
  [#2][2147483647][\cref@result]#1%
}
\makeatother

\crefname{lemma}{Lemma}{Lemmas}
\Crefname{lemma}{Lemma}{Lemmas}
\crefname{theorem}{Theorem}{Theorems}
\Crefname{theorem}{Theorem}{Theorems}
\crefname{proof}{Proof}{Proofs}
\Crefname{proof}{Proof}{Proofs}
\crefname{algorithm}{Algorithm}{Algorithms}
\Crefname{algorithm}{Algorithm}{Algorithms}

\ifpdf
  \DeclareGraphicsExtensions{.eps,.pdf,.png,.jpg}
\else
  \DeclareGraphicsExtensions{.eps}
\fi

\newsiamremark{remark}{Remark}
\newsiamremark{hypothesis}{Hypothesis}
\crefname{hypothesis}{Hypothesis}{Hypotheses}
\newsiamthm{claim}{Claim}

\headers{Fault Oblivious Eigenvalue Solver}{Mukherjee, Kang, Gleich, Sameh, and Grama}

\title{Fault Oblivious Eigenvalue Solver\thanks{
\funding{This work was funded by the Department of Energy(DOE DE-SC0023162 Sparsitute) and the NSF (IIS-2007481).}}}

\author{Jayanta Mukherjee\thanks{Department of Computer Science, Purdue University, West Lafayette, IN 
  (\email{jmukher@purdue.edu}).}
\and Xuejiao Kang\thanks{Facebook 
  (\email{xue@meta.com})}
\and David F. Gleich\thanks{Department of Computer Science, Purdue University, West Lafayette, IN 
  (\email{dgleich@purdue.edu})}
\and Ahmed Sameh\thanks{Department of Computer Science, Purdue University, West Lafayette, IN 
  (\email{sameh@purdue.edu})}
\and Ananth Grama\thanks{Department of Computer Science, Purdue University, West Lafayette, IN 
  (\email{ayg@cs.purdue.edu})}}
\usepackage{amsopn}

\newcommand{\tA}{\tilde{A}}
\newcommand{\tB}{\tilde{B}}
\newcommand{\tx}{\tilde{x}}

\ifpdf
\hypersetup{
  pdftitle={Fault Oblivious Eigenvalue Solver},
  pdfauthor={Jayanta Mukherjee, Xuejiao Kang, David F. Gleich, Ahmed Sameh, and Ananth Grama}
}
\fi

\begin{document}

\maketitle

\begin{abstract}
Eigenvalue problems serve as fundamental substrates for applications in large-scale scientific simulations and machine learning, often requiring computation on massively parallel platforms. As these platforms scale to hundreds of thousands of cores, hardware failures become a significant challenge to reliability and efficiency. In this paper, we propose and analyze a novel fault-tolerant eigenvalue solver based on erasure-coded computations -- a technique that enhances resilience by augmenting the system with redundant data \textbf{a priori}.  This transformation reformulates the original eigenvalue problem as a generalized eigenvalue problem, enabling fault-oblivious computation while preserving numerical stability and convergence properties. We formulate the augmentation scheme, establish the necessary conditions for the encoded blocks, and prove the relationship between the original and transformed problems. We implement an erasure-coded TraceMin eigensolver and demonstrate its effectiveness in extracting eigenvalues in the presence of faults. Our experimental results show that the proposed solver incurs minimal computational overhead, maintains robust convergence, and scales efficiently with the number of faults, making it a practical solution for resilient eigenvalue computations in large-scale systems.
\end{abstract}



\section{Introduction}
Eigenvalue problems are computationally intensive and arise in various domains, often requiring solutions on scalable parallel and distributed platforms. The high complexity and massive scale of such platforms makes fault tolerance a critical consideration. Traditional parallel computations typically rely on checkpoint-restart mechanisms for fault tolerance. However, these techniques present two major challenges: (i) they require consistent checkpoints, which may incur significant overhead due to rollback -- especially in scalable parallel programs that aim to minimize global synchronization; and (ii) they demand substantial I/O capacity and bandwidth to store checkpoints in persistent storage or sufficient interconnect bandwidth for in-memory checkpoints. 

An alternative approach to fault tolerance involves detecting and mitigating failures using active replicas in conjunction with a consensus procedure. Active replicas are typically used in real-time systems, where worst-case execution times must be guaranteed, and rollback/ replay schemes may violate such guarantees. However, active replicas have high resource overhead, as each computation must be executed by $s+1$ replicas to tolerate $s$ faults.

Erasure coding is frequently used in storage systems to provide efficient and scalable fault tolerance by adding appropriately coded redundancies to overcome data erasures. These codes can be conceptualized as multiplying a data vector of size $n$ by a coding matrix consisting of $n$ columns and $m > n$ rows. If any $n$ of the $m$ rows of the coding matrix are guaranteed to be linearly independent, the vector resulting from this matrix-vector product contains sufficient redundancies to tolerate up to $m-n$ erasures.
Specifically, in the event of up to $m-n$ erasures, the remaining $n$ elements can be used with the corresponding $n \times n$ non-singular matrix corresponding to the non-erased rows of the coding matrix to recover the original $n$ data items. Building on this concept of erasure-coded storage, in prior work, we introduced the notion of erasure-coded computation {\em} for solving linear systems. In this work, the input problem is augmented with suitably coded blocks, the augmented problem instance is solved on a faulty parallel platform in a fault-oblivious manner, and the solution is recovered from the results on the non-faulty processors using an inexpensive procedure~\cite{kang_linear_erasure,joyceDynamic,  Zhu2014ErasureCF}.



In this paper, we present a novel formulation of erasure-coded computations for solving eigenvalue problems. In contrast to solving linear systems, naively augmenting the input matrix with row and column blocks alters the spectrum, with no known inexpensive methods of recovering the original eigenvalues.

We focus here on faults that cause erasure (or deletion) of data, along with fail-stop failure of the processors on which the corresponding part of the computation may execute. We present a novel erasure-coded computation scheme for fault-tolerant solution of eigenvalue problems, $Ax = \lambda x$, for given matrix $A$. 
Unlike linear systems, naively adding a coding block to a given matrix $A$ changes its eigenvalues, and there are no known computationally inexpensive ways of recovering the original eigenvalues from these perturbed eigenvalues. To address this, we transform the original eigenvalue problem to an equivalent (in terms of eigenvalues) generalized eigenvalue problem $\Tilde{A}\Tilde{x} = \lambda \Tilde{B} \Tilde{x}$, where $\tilde{A}$ is an augmented form of matrix $A$ and $\tilde{B}$ is an encoded identity matrix. 
The resulting augmented problem can be solved using most off-the-shelf eigensolvers in a fault-oblivious manner; i.e., in the event of a fail-stop failure, the remaining processors simply proceed with their computation oblivious to the faults.

We present detailed proofs establishing the equivalence between the eigenvalues of the original problem and those of a reformulated generalized eigenvalue problem (Theorem-\ref{th:eigenvalue_equivalence}). We analyze the impact of faults on the generalized problem and provide methods to recover eigenvalues in the presence of such faults. Although our primary fault model assumes fail-stop failures - analogous to erasures in storage - other fault types (e.g., transient or soft faults) can be handled similarly using predicates to detect and isolate faulty states.

We solve the reformulated problem using TraceMin \cite{samehParallel, 2000sameh, 1982Sameh} and show that our approach achieves a convergence behavior comparable to the fault-free case, even under random erasures (randomly selected row-column removals).

Our experiments quantify the impact of different fault models on convergence, the overhead introduced by our erasure-coded eigensolver, the benefits of optimizations, and sensitivity to various parameters, including approximation levels in adaptive code construction. These results demonstrate the efficacy of our approach in solving eigenvalue problems in fault-prone computing environments.

In summary, our contributions are as follows.
\begin{enumerate}[label=(\roman*)] 
\item A novel reformulation of the eigenvalue problem as an equivalent generalized eigenvalue problem with augmented matrices, enabling fault-oblivious computation. 
\item Low-overhead erasure-coding schemes for fault-tolerant eigenvalue computations. 
\item Application of erasure coding to both the Power Method and TraceMin, demonstrating the generality of our approach. 
\item Performance comparison with checkpoint-restart techniques, highlighting the advantages of our erasure-coded solver. 
\item Analysis of both single- and multifault scenarios, demonstrating robustness across diverse fault conditions. 
\end{enumerate}
\section{Related Research}\label{sec:related}
Accurate computation of eigenvalues is an essential part of ML applications operating in hardware environments ranging from embedded devices in harsh environments to data-center scale solvers. 
Fault tolerance techniques in these environments can be classified into two broad categories: system-supported and algorithm-based. System-supported methods include checkpoint-restart~\cite{dongarra_check}, active replicas \cite{activereplica_ipdps}, and deterministic replay~\cite{Deterministic_replay07}. Checkpoint-restart techniques periodically save the application state into persistent storage (disks or replicated in-memory). This requires the identification of consistent checkpoints and the capacity for persistent storage in terms of space and bandwidth. Active replicas execute computations on multiple processors -- these replicas are monitored for potential faults, and a consensus protocol identifies fault-free executions.  
Algorithm-based fault tolerance (ABFT) methods modify the base algorithm to embed redundant computations to render the overall computation resilient to faults \cite{BOSILCA2009410, Bridges2012FaulttolerantLS, chen2009, chen_dongarra2008, Fault_chen05, huang1984, LUK1988172}. While ABFT methods often have advantages over system-supported methods in resource overheads, they must be specifically designed for each algorithm, leverage specific aspects of the algorithm and fault characteristics, and typically require intricate correctness proofs. 
Our method can be viewed as the first method for fault-tolerant eigenvalue computations in the broad class of ABFT methods. It leverages results from coding theory (sparse codes), linear algebra (augmented problem formulation, correctness proofs), randomized techniques (leverage score sampling), and efficient solvers (conditioning, convergence), to deliver a novel high-performance fault-tolerant eigensolver.
\section{Erasure Coded Eigenvalue Solver}
\label{sec:alg}


Our proposed solution adds redundant rows and columns to the matrix to render it tolerant to faults with any solver. In contrast to linear system solvers, adding a row (or column) to a matrix, even if it is in the row-subspace of the matrix, changes its eigenvalues. For this reason, a key challenge for us is the reformulation of the eigenvalue problem so that the addition of a coding block still allows for inexpensive recovery of the original eigenvalues. In this section, we present a novel reformulation of the eigenvalue problem, along with recovery algorithms.

\subsection{Formulating an Erasure Coded Eigensolver}\label{sec:formulation}

A standard symmetric eigenvalue problem can be written as:
\begin{equation}\label{eq:eig}
    A x = \lambda x,
\end{equation}
where $A \in \mathbb{R}^{n\times n}$ is a symmetric matrix, $\lambda$ is an eigenvalue, and $x$ is the corresponding eigenvector. A generalized eigenvalue problem is
\begin{equation}\label{eq:geneig}
    A x = \lambda B x,
  \end{equation}
 which reduces to the original problem when $B$ is the identity matrix ($I_n$) of size $n\times n$. 

Let $E$ be an $n \times k$ matrix we call the \emph{coding matrix}. In our design, the coding matrix $E$ should have Kruskal row rank of $k$ \cite{Kruskal1977ThreewayAR} to ensure tolerance of up to $k$ faults that occur anywhere in the system. We derive the \textit{recovery equation} for our erasure coding as:
\begin{equation}\label{eq:recovery}
    x^{*} = x + E r .   
\end{equation}
Here, $x$ is the eigenvector of the original eigenvalue problem \ref{eq:eig} of dimension $n$, and $r$ is the redundant part (related to fault-tolerance) of dimension $k$
We can substitute the recovery equation \ref{eq:recovery} into original eigenvalue system \ref{eq:eig} to get
\begin{equation}\label{eq:rec1}
    A(x + E r) = \lambda (x + Er)
\end{equation}

Recall that, Kruskal row rank $k$ of a matrix implies that any subset of $k$ rows of the matrix is guaranteed to be linearly independent. Now, as we want some redundancy, we add another set of linear constraints.
\begin{equation}\label{eq:redun}
    E^{T}[A x^{*} = \lambda x^{*}]
\end{equation}
Putting  \ref{eq:rec1} and \ref{eq:redun} together we have,
\begin{equation}\label{eq:group}
    \left\{ \begin{aligned} 
  A (x + E r) &= \lambda (x + Er)\\
  E^{T} A (x + E r) &= \lambda E^{T} (x + E r)
\end{aligned} \right.
\end{equation}
After grouping the terms of the above \ref{eq:group}, we arrive at the augmented eigenvalue problem:
\begin{equation}\label{eq:augmented}
    \begin{bmatrix} A & AE \\ E^{T}A & E^{T}AE \end{bmatrix} \begin{bmatrix} x \\ r \end{bmatrix} = \lambda \begin{bmatrix} I & E \\ E^{T} & E^{T}E \end{bmatrix} \begin{bmatrix} x \\ r \end{bmatrix}
\end{equation}

We more compactly write this as the following generalized eigenvalue problem:
 \begin{equation}\label{eq:generasureeig}
    \Tilde{A} \Tilde{x} = \lambda \Tilde{B} \Tilde{x},
  \end{equation}
 where, $\Tilde{A}$, $\Tilde{B}$ are the augmented matrices
  \begin{equation}
  \label{eq:augmented2}
  \begin{aligned}
  \Tilde{A} = \begin{bmatrix} A & AE \\ E^{T}A & E^{T}AE \end{bmatrix}, \quad  
  \Tilde{B} =
  \begin{bmatrix} I & E \\ E^{T} & E^{T}E \end{bmatrix},
  \quad
  \Tilde{x} = \begin{bmatrix} x \\ r \end{bmatrix}
  \end{aligned}
  \end{equation}
and $\Tilde{x}$ is the eigenvector of the augmented system.


This generalized eigenvalue problem is singular. In this case, this is because there exists a vector $\tilde{x}$ such that $\tA \tilde{x} = 0$ and $\tB \tilde{x} = 0$ for any vector $\tilde{x}$ in the joint null-space of $\tA$ and $\tB$. 

\begin{lemma}[Null Space in Augmented Generalized Eigenvalue System]
\label{th:nullspace}
The matrices $\Tilde{A}$ and $\Tilde{B}$ from \eqref{eq:augmented2} have a joint null space $\begin{bmatrix}
    E \\ -I_{k}
\end{bmatrix}$
\end{lemma}

\begin{proof}
    Recall $\tA$ and $\tB$ are $(n+k) \times (n+k)$. Note that $\tB$ has rank $n$ since the last $k$ rows are a linear combination of the previous $n$, and the first $n$ rows have an identity block. Therefore, the joint null space must have rank at most $k$. The matrix $\begin{bmatrix}
    E \\- I_{k}
\end{bmatrix}$ has rank $k$ and 
\begin{equation}\label{eq:null}
    \begin{aligned}
        \Tilde{A}\begin{bmatrix}
    E \\ -I_{k}
\end{bmatrix} = 0 , \quad&  
        \Tilde{B}\begin{bmatrix}
    E \\ -I_{k}
\end{bmatrix} = 0.
    \end{aligned}
\end{equation}
\end{proof}

To make the structure inside the singular pencil clear, we next show that the pencil $(\tA, \tB)$ is strictly equivalent to a simple case related to the eigenvalues of $A$ alone.

\begin{theorem}[Matrix Pencil Equivalence]
Let $\tA, \tB$ be from \eqref{eq:augmented}. Then the pencil $\tA - \lambda \tB$ is strictly equivalent to the pencil 
\[ \begin{bmatrix} 0 & 0 \\ 0 & A (I + E E^T) \end{bmatrix} - \lambda \begin{bmatrix} 0 & 0 \\ 0 & (I + E E^T) \end{bmatrix}. \]  
Since the matrix $(I + EE^T)$ is non-singular for any $E$, a subset of eigenvalues and eigenvectors of the generalized eigenvalues of $(\tA, \tB)$ will be related to the eigenvalues of $A$. 
\end{theorem}
\begin{proof}
Let $M = \begin{bmatrix}E & I \\ -I_k & E\end{bmatrix}$, $M$ is invertible with the last $n$ columns orthogonal to the first $k$. Then $M^{-1} (\tA - \lambda \tB) M$ is a strict equivalence transformation and 
\[ M^{-1} (\tA - \lambda \tB) M = \begin{bmatrix} 0 & 0 \\ 0 & A (I + E E^T) \end{bmatrix} - \lambda \begin{bmatrix} 0 & 0 \\ 0 & (I + E E^T) \end{bmatrix}. \]
For a quick understanding of why this form follows, note that $M = \begin{bmatrix} Z & Y\end{bmatrix}$ where $Z$ is the basis for the joint null space of $A$ and $B$ and $Z^T Y = 0$, so these are orthogonal subspaces. A straightforward elementary justification is possible through the closed form calculation of 
\[ M^{-1} = \begin{bmatrix} E^T (I + EE^T)^{-1} & -(I + E^T E)^{-1} \\ 
                    (I + EE^T)^{-1} & E(I + E^T E)^{-1} \end{bmatrix}. \] 
The two trickier steps are showing $E^T (I + E E^T)^{-1} + (I + E^T E)^{-1} E^T = 0$ (which shows up in the $2,1$ block) and  $(I + E E^T)^{-1} + E (I + E^T E)^{-1} E^T = I$ (which shows up in the $2,2$ block). It suffices to look at the computation for $\tA$ alone as the structure of $\tA$ and $\tB$ are the same with $A$ replaced by an identity block. 

For the trailing size $n$ block, we then have the generalized problem $A (I + EE^T) y = \lambda (I+EE^T) y$, which is a similarity transform of the eigenvalues of $A$ through $(I + EE^T)^{-1} A (I + EE^T)$. 
\end{proof}





Since the singular subspace of the generalized eigenvalue problem can have \emph{any} eigenvalues when treated computationally the augmented eigenvalue system (as shown in \ref{eq:generasureeig}) will have some spurious eigenvalues that correspond to the vectors in the null-space of $\begin{bmatrix} I \thickspace E \end{bmatrix}$. These can be easily detected.

\paragraph*{\textbf{Example}} We illustrate this encoding and detection of spurious eigenvalues using a simple example. Let: 
\[ A = \begin{bmatrix}
    2 & -1 & 0 & 0 \\
-1 & 2 & -1 & 0 \\
0 & -1 & 2 & -1 \\
0 & 0 & -1 & 2 \\
\end{bmatrix} 
\quad 
E = 
\begin{bmatrix} 
0.98 & 0.42 \\
0.13 & 0.39 \\
0.53 & 0.85 \\
0.87 & 0.93 \\
\end{bmatrix}
\]
The eigenvalues of $A$ are (to four digits) 
\[ \begin{array}{cccc} 
0.382 & 1.382 & 2.618 & 3.618 
\end{array} \]
One computation of the generalized eigenvalues of eigenvectors of $\Tilde{A}$ and $\Tilde{B}$ yielded: 
\[ \begin{array}{cccccc}
    0.382 & 0.6012 & 1.382 & 2.618 & 3.6031 & 3.618
\end{array}
\] 
\[ \begin{bmatrix} 
-0.1696 & 0.1728 & -0.2473 & -0.258 & -0.4414 & 0.4409 \\
-0.5019 & 0.3572 & 0.2104 & 0.1836 & 0.2757 & -0.2745 \\
-0.8198 & 0.7163 & 0.7065 & 0.3035 & 0.3842 & -0.3852 \\
-0.7187 & 0.7106 & 0.6918 & 0.3455 & 0.1653 & -0.1643 \\
-0.3552 & 0.2522 & 0.4895 & 0.8502 & 0.879 & -0.8793 \\
1.0 & -1.0 & -1.0 & -1.0 & -1.0 & 1.0 \\
\end{bmatrix} 
\]
(each column of the matrix is an eigenvector with the associated eigenvalue above it).
The two spurious eigenvalues can be detected by computing $\begin{bmatrix} I & E \end{bmatrix} r$ for each eigenvector and checking for zero (or a very small norm, more generally). For example, 
\[ \begin{bmatrix} 
0.1728 \\
0.3572 \\
0.7163 \\
0.7106 \\
\end{bmatrix}
+ E \begin{bmatrix}
    0.2522 \\ -1
\end{bmatrix} = 0
\]

We summarize as follows. Given any generalized eigenvector $\tx = \left[\begin{smallmatrix} x \\ r \end{smallmatrix}\right]$ of $(\tA, \tB)$, then compute $x + Er$ and if this is zero, then it is a spurious solution. 


In principle, the singularity gives us the fault tolerance we seek. The eigenvalue problem remains valid with up to $k$ elements of $x$ set to arbitrary values. However, the singular pencil poses a variety of computational issues that are difficult to resolve. Consequently, we use this idea for inspiration and consider ideas more closely related to those used in~\cite{joyceDynamic}. This involves 
replacing the erased row-column pairs with the corresponding coding blocks from the augmented matrices $\Tilde{A}$ and $\Tilde{B}$ when a fault occurs.

\subsection{A More Practical Fault Handling and Solution Recovery Idea}\label{sec:solrecovery}
The idea for our more practical solver is that we treat the blocks $R = AE$, $S = E^T A E$, $T = E^T E$ as \emph{redundancy} that can be used whenever it is needed by a fault. (Recall that we target symmetric $A$ so $(AE)^T = E^T A$.) Then, when fault occur, we substitute these into the system for the missing parts. 
Then we initiate an eigensolver on $A$. 
In the event of a fault, a subset of rows or columns in the matrices $A$ and $B$ are ``erased'' in that we lose access to them. Note that we consider only fail-stop failures; that is, in the event of a failure, a compute node halts both computation and communication.

After erasures, we reconstitute the matrices $A'$ and $B'$ by substituting the erased rows and columns with the corresponding rows and columns from the blocks of $E, R, S, T$, as shown in Figure \ref{fig:schematic}, and proceed with the eigenvalue solver. For example, if rows $i_1$ to $i_2$ of matrix $A$ are deleted, we replace these rows and their corresponding columns in the erased system (represented as $A_r$) with the pre-computed erasure-coded blocks from $E^{T}A$ and $A E$, respectively, to preserve symmetry. Similarly, the erased rows and columns of matrix $B$, are replaced with the corresponding rows and columns of $E^{T}B$ and $BE$. The intersecting square-block from ($i_{1},i_{1}$) to ($i_{2},i_{2}$) of $B$ is replaced by the square block of size $(i_{2}-i_{1}) \times (i_{2}-i_{1})$ of $E^{T}BE$ of the augmented matrix $\Tilde{B}$ as illustrated in Figure~\ref{fig:schematic}. Once the matrices have been reconstituted, the solver resumes execution. 

\begin{figure}[h! tbp]
\centering
\includegraphics[width=.99\linewidth]{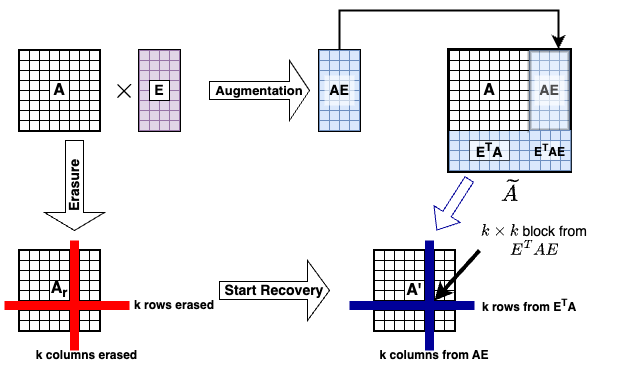}
  \caption{Erasure-Coded TraceMin Solver Flow. 
  }
\label{fig:schematic}
\end{figure}

\paragraph*{\textbf{Example}}
Continuing our previous example, we have $A$ and $E$ as before. We use the recovery information 
\[ 
R^T = \begin{bmatrix}
1.83 & -1.25 & 0.06 & 1.21 \\
0.45 & -0.49 & 0.38 & 1.01 \\
\end{bmatrix}
\quad 
S = \begin{bmatrix}
    2.7154 & 1.4574 \\
1.4574 & 1.2602 \\
\end{bmatrix}
\quad
T = \begin{bmatrix}
    2.0151 & 1.7219 \\
1.7219 & 1.9159 \\
\end{bmatrix}. 
\]
Suppose as we were solving an eigenvalue problem $Ax = \lambda x$, we had a failure in the 3rd row. Then we'd assemble the new matrix by replacing the 3rd row and column with information from the 1st row of $R^T$. (There is nothing special about this row, we could have used the 2nd as well, but our convention is to use the first unused row.) This yields the new \emph{generalized} eigenvalue problem with $A'x = \lambda B'x$
\[
A' = \begin{bmatrix}
    2.0 & -1.0 & 1.83 & 0.0 \\
-1.0 & 2.0 & -1.25 & 0.0 \\
1.83 & -1.25 & 2.7154 & 1.21 \\
0.0 & 0.0 & 1.21 & 2.0 \\
\end{bmatrix}
\quad 
B' = \begin{bmatrix}
    1.0 & 0.0 & 0.98 & 0.0 \\
0.0 & 1.0 & 0.13 & 0.0 \\
0.98 & 0.13 & 2.0151 & 0.87 \\
0.0 & 0.0 & 0.87 & 1.0 
\end{bmatrix} 
\]
The generalized eigenvalues and vectors of this new system are 
\[ \begin{array}{cccccc}
    0.382 & 1.382 & 2.618 & 3.618
\end{array}
\] 
\[ \begin{bmatrix} 
0.7405 & 1.2889 & 1.2889 & -0.7405 \\
-0.454 & 0.4629 & -0.2806 & -0.749 \\
-1.1349 & -0.7014 & -0.7014 & 1.1349 \\
0.6156 & 0.0087 & 1.2117 & -1.3591 \\
\end{bmatrix}.  \] 
To recover the eigenvectors of the original system, we need to use the ``recovery equation $x + Er$''. In this case -- when applied to each eigenvector as a matrix equation -- this becomes
\[  \begin{aligned} 
& \begin{bmatrix} 
0.7405 & 1.2889 & 1.2889 & -0.7405 \\
-0.454 & 0.4629 & -0.2806 & -0.749 \\
0 & 0 & 0 & 0 \\
0.6156 & 0.0087 & 1.2117 & -1.3591 \\
\end{bmatrix}  + \begin{bmatrix}  0.98 \\ 
 0.13\\ 
 0.53\\ 
 0.87 \end{bmatrix} 
 \begin{bmatrix} 
 -1.1349 & -0.7014 & -0.7014 & 1.1349
 \end{bmatrix} \\
 & = 
 \begin{bmatrix} 
 -0.3717 & 0.6015 & 0.6015 & 0.3717 \\
-0.6015 & 0.3717 & -0.3717 & -0.6015 \\
-0.6015 & -0.3717 & -0.3717 & 0.6015 \\
-0.3717 & -0.6015 & 0.6015 & -0.3717 \\
\end{bmatrix}
\end{aligned}
 \]
Thus, we \emph{recover} the original eigenvectors from the generalized eigensystem.

To describe and analyze this process formally, we will permute the system so that $C$ is the part of $A$ that is non-faulty and $c$ are the corresponding coefficients of $x$ that were not erased.  The faulty components will be placed at the \emph{end} of the matrix for analysis as in 
\begin{equation} \label{eq:faulty-partitioning} Ax = \lambda x \quad \implies \quad \begin{bmatrix}
    C & F_1 \\
    F_1^T & F_2^T
\end{bmatrix} \begin{bmatrix} 
c \\ 
f
\end{bmatrix} 
= \lambda 
\begin{bmatrix} 
c \\ 
f
\end{bmatrix}. \end{equation} The matrices $F_1^T$ and $F_2^T$ correspond to the erased rows, as well as the components $f$. Suppose there are $\ell \le k$ failed or erased rows. The matrix $C$ just consists of the rows and columns of $A$ that are still correctly executing. We are going to replace the rows and columns corresponding to $F_1$ with $F_2$ with the redundant data. Let  
\begin{equation} \label{eq:reconstituted-matrices} R^T = E^T A =\begin{bmatrix} Z^T\\ Y^T \end{bmatrix} = \begin{bmatrix} Z_C^T & Z_F^T \\ Y_C^T & Y_F^T \end{bmatrix} \quad \text { and } \quad E^T = \begin{bmatrix} E_Z^T \\ E_Y^T \end{bmatrix} = \begin{bmatrix} E_{ZC}^T & E_{ZF}^T \\ E_{YC}^T & E_{YF} \end{bmatrix}  \end{equation}
be a partition of the redundancy data into the first $\ell$ rows of $R^T$ and the later rows of $R$, along with a corresponding partition of the first $\ell$ rows of the matrix $E^T$; the column partition is given by the correct and faulty rows and columns of the matrix $A$. Also, let $S_Z$ and $T_Z$ be the first $\ell$ rows and columns of $S$ and $T$, respectively.  
Then the reconstituted eigensystem is  
\begin{equation}  \label{eq:reconstituted-full} \begin{bmatrix} 
C & Z_C \\ 
Z_C^T & S_Z \end{bmatrix}
\begin{bmatrix}
    c \\ r 
\end{bmatrix} = 
\lambda 
\begin{bmatrix} 
I & E_{CZ} \\ 
E_{CZ}^T & T_Z \end{bmatrix}
\begin{bmatrix}
    c \\ r 
\end{bmatrix}.
\end{equation}
Here $r$ are components of the solution that are not part of the original solution but need to be included into the solve. (The values in $r$ correspond to the 3rd row of our eigenvector matrix above.)

One may question the need for a coding block, as opposed to replacing erased parts of matrix $A$ with corresponding blocks of $A$ itself. The problem with this approach is that the matrix $A$ will have to be replicated $(k+1)$-fold to tolerate $k$ faults. This is infeasible for problems that operate close to limit of memory.

We now show that this process leaves the eigenvalues unchanged and permits us to recover the eigenvectors. We also simplify notation and assume we use \emph{all} of the redundant data. We can do this without loss of generality because $Z_C, S_Z,$ and $T_Z$ \emph{only} use the information corresponding to the first $\ell$ columns of $E$.

\begin{theorem}
[Eigenvalue Equivalence]
\label{th:eigenvalue_equivalence}
Consider a eigenvalue problem $Ax = \lambda x$ where $A$ is symmetric. Let $E$ be a coding matrix with Kruskal row rank $k$. Let $R = AE$, $S = E^T A E$, and $T = E^T E$. Suppose there are $k$ faults, which result in loss of $k$-rows and columns of $A$. Without loss of generality, we have the problem permuted so these are the last $k$ rows and columns as in \ref{eq:faulty-partitioning}. Let 
\[ E = \begin{bmatrix} E_C \\ E_F \end{bmatrix} \quad \text{ and } \quad R = \begin{bmatrix} R_C \\ R_F \end{bmatrix} \]  correspond to a partition of the data in $E$ and $Z$ into the same set of \emph{correct} and \emph{faulty} rows. Then the generalized eigenvalues of $A' y = \lambda' B' y$, that is
\begin{equation} \label{eq:reconstituted}
\begin{bmatrix} C & R_C \\ R_C^T & S \end{bmatrix} \begin{bmatrix} c \\ r \end{bmatrix} = \lambda' \begin{bmatrix} I & E_C \\ E_C^T & T \end{bmatrix} \begin{bmatrix} c \\ r \end{bmatrix},
\end{equation} 
are identical to those of $A$. Also, the eigenvectors are related by a simple linear transformation. 
\end{theorem}

\begin{proof}
We will show that if $y = \left[ \begin{smallmatrix} c \\ r \end{smallmatrix} \right] $ is a generalized eigenvector of \ref{eq:reconstituted} with eigenvalue $\lambda'$ then 
\begin{equation} 
 v = \begin{bmatrix} v_c \\ v_f \end{bmatrix} = \begin{bmatrix} c + E_{C}r \\ E_{F} r \end{bmatrix} = \underbrace{\begin{bmatrix} I & E_C \\ 0 & E_F \end{bmatrix}}_{= M} \begin{bmatrix} c \\ r \end{bmatrix}  \end{equation}
is an eigenvector of $A$ with the same eigenvalue. Because $E$ has Kruskal rank $k$, the matrix $E_F$ is $k \times k$ and invertible, so the transformation matrix $M$ is invertible, which is the crux of the argument. Let $A' y = \lambda' B' y$ be the generalized eigenvalue problem~\eqref{eq:reconstituted}. Since $My = v$ with $A' = M^{T} A M$ and $B' = M^T M$ then we arrive at $Av = \lambda v$.  
\end{proof}

\subsection{Implementing An Erasure-Coded Power Method}
For the erasure-coded variant of the Generalized Power Method, we compute the erasure coding matrix $E$. We then derive the augmented coding blocks $R, S, T$ from matrix $A$ and $E$.   
We now show two different implementations of a power method. The first is one that exhibits the ideas clearly. The second is one that illustrates how a practical implementation might function. To be more concrete, in the first, we use matrices $A$ and $B$ to represent the operations that would be performed with modification to the matrices. In the second, we'll illustrate how these operations can be optimized using the structure of the problem. 

The preceding theory dealt with the \emph{eigensystems} that arise. When we move to these algorithms, we need to consider how to reinitialize the values of $X_F$ after a fault occurs. We assume that the values $X_F$ are totally lost. First, if an implementation periodically computes and retains the quantity $E^T X$, then we can always extract the faulty missing elements from this matrix -- this is a simple calculation related to the ideas from~\cite{zhugleich2017}.  However, since this has the flavor of checkpointing, we do not do it. Although it may seem like we should be able to utilize the non-faulty values $X_C$ to derive improved values for $X_F$, we were unable to make this idea successful. Consequently, when a fault occurs, we replace the values $X_F$ with random data. This appears to work the best in our experiments. 

\begin{algorithm}[t]
\caption{Erasure-Coded Generalized Block Power Method with QR Subspace Iteration}\label{algo:gen_block_power_qr}
\begin{algorithmic}[1]
\Require Symmetric matrix $A \in \mathbb{R}^{n \times n}$, initial matrix $X_0 \in \mathbb{R}^{n \times k}$, tolerance $\varepsilon$, redundancy information $E, R, S, T$
\Ensure Approximate eigenpairs $(\Lambda, X)$ where $A X \approx \Lambda B X $
\State Initialize $X \gets X_0$
\State Initialize $B \gets I$
\For{$i = 1, 2, ...$ until convergence}
    \If{Fault Occurs}
            \State Replace the faulty rows and columns of $A$ with any unused rows and columns of $R$ and the intersection of rows and columns with block $S$ to reconstitute $A'$
            \State Replace the faulty rows and columns of $B$ with any unused rows and columns of $E$ and the intersection of rows and columns with block $T$ to reconstitute $B'$.
            \State{Use the reconstituted $A'$ and $B'$ as $A$ and $B$ respectively and continue the solver; $A \gets A'$; $B \gets B'$ }
            \State{Set any failed rows of $X$ to random data and orthogonalize with a QR factorization.}
    \EndIf
    \State $Y \gets A X$
    \State $Z \gets B^{-1} Y$ \Comment{Solve $B Z = A X$}
    \State $[Q, \sim] \gets \text{qr}(Z)$ \Comment{Thin QR factorization}
    \State $A_Q \gets Q^\top A Q$, \quad $B_Q \gets Q^\top B Q$
    \State $[U, D] \gets \text{eig}(A_Q, B_Q)$ \Comment{Solve small generalized eigenproblem}
    \State Sort eigenvalues $\Lambda$ in $D$ in descending order; reorder $U$ accordingly
    \State $X_{\text{new}} \gets Q U$
    \State Compute residual $r \gets \| A X_{\text{new}} - B X_{\text{new}} \Lambda \|_F$
    \State Compute relative residual $r_{\mathrm{rel}} \gets \frac{r}{\|A\|_F}$
    \If{$r_{rel} < \varepsilon$}
        \State \textbf{break}
    \EndIf
    \State $X \gets X_{\text{new}}$
\EndFor
\Return $X, \Lambda$
\end{algorithmic}
\end{algorithm}

\begin{algorithm}[t]
\caption{Erasure-Coded Generalized Block Power Method with QR Subspace Iteration }\label{algo:gen_block_power_qr_bsolve}
\begin{algorithmic}[1]
\Require Symmetric matrix $A \in \mathbb{R}^{n \times n}$, initial orthogonal matrix $X_0 \in \mathbb{R}^{n \times k}$, tolerance $\varepsilon$, redundancy information $E^{(0)}, Z^{(0)}, S^{(0)}, T^{(0)}$ up to $k$ faults.
\Ensure Approximate eigenpairs $(\Lambda, X)$ where $A X \approx \Lambda B X $
\State Initialize $X \gets X_0$, $\mathcal{F} = \emptyset$ (the set of faulty indices), 
\State $E \gets \text{getErasureCodedMatrix}(n,k)$ \Comment{generate encoding matrix}
\State $R \gets A \cdot E,\; S \gets E^\top A \cdot E,\; T \gets E^\top E$ \Comment{generate coding blocks}
\For{$i = 1, 2, ...$ until convergence}
    \If{a new fault occurs}
      \State Let $\mathcal{F}_{\text{new}}$ be the set of new faulty indices 
      \State Add the faculty indices to $\mathcal{F}$ (unless $\mathcal{F}$ exceeds $k$ indices, in which case terminate with an error that the fault capacity was exceeded). 
      \State Set $E_C, Z_C$ to the first $|\mathcal{F}|$ columns of $E^{(0)}, Z^{(0)}$
      \State Set $S, T$ to be the first $|\mathcal{F}| \times |\mathcal{F}|$ block of $S^{(0)}, T^{(0)}$
      \State Compute a factorization of $T - E_C^T E_C$ for applications of $(T - E_C^T E_C)^{-1}$
      \State Set $X(\mathcal{F}_{\text{new}},:) = $ random normal entries 
      \State Set $A(\mathcal{F}_{\text{new}}, :) = A(:, \mathcal{F}_{\text{new}}) = 0$ (or implicitly via the fault)
    \EndIf 
    \If{$\mathcal{F}_{\text{new}}\neq\emptyset$}
        \State Update fault set $\mathcal{F}\!\gets\!\mathrm{sort}(\mathcal{F}\cup \mathcal{F}_{\text{new}})$; $EF\!\gets\!E[\mathcal{F},:]$, $G\!\gets\!\mathrm{chol}(EF^\top EF)$
    \EndIf
    \State $Y \gets \text{Aop}(A,R,S,X,\mathcal{F})$
    \State $Z \gets \text{invBop}(G,E,EF,Y,\mathcal{F})$ \Comment{Solve $B Z = A X$}  
    \State $[Q, \sim] \gets \text{qr}(Z)$ \Comment{Thin QR factorization}
    \State $AQ \gets Q^\top \cdot \text{Aop}(A,R,S,Q,\mathcal{F})$
            \quad $BQ \gets Q^\top \cdot \text{Bop}(E,T,Q,\mathcal{F})$
    \State $[U, D] \gets \text{eig}(A_Q, B_Q)$ \Comment{Solve small generalized eigenproblem}
    \State Sort eigenvalues $\Lambda$ in $D$ in descending order; reorder $U$ accordingly
    \State $X_{\text{new}} \gets Q U$
    \State Compute residual $r \gets \| A X_{\text{new}} - B X_{\text{new}} \Lambda \|_F$
    \State Compute relative residual $r_{rel} \gets \frac{r}{\| A \|_F}$
    \If{$r_{rel} < \varepsilon$}
        \State \textbf{break}
    \EndIf
    \State $X \gets X_{\text{new}}$
\EndFor
\Return $X, \Lambda$ 
\end{algorithmic}
\end{algorithm}

\begin{algorithm}[tb]
\caption{Erasure-Coded Operator Application and Inversion}
\label{alg:ec-ops}
\begin{algorithmic}[1]
\Require Mode $\in \{\texttt{Aop}, \texttt{Bop}, \texttt{invBop}\}$; matrices $A, E, Z, S, T, EF$; input $X$; faulty indices $F$; Cholesky factor $G = \mathrm{chol}(EF^{\top}EF,\text{upper})$
\Ensure Output $Y$

\Function{Aop}{$A, Z, S, X, \mathcal{F}$}
    \State $Y \gets A X$
    \State $Y \gets Y + Z\, X(\mathcal{F},:)$
    \State $Y(\mathcal{F},:) \gets Y(\mathcal{F},:) + Z^{\top} X$
    \State $Y(\mathcal{F},:) \gets Y(\mathcal{F},:) + (S - Z(\mathcal{F},:) - Z(\mathcal{F},:)^{\top})\, X(\mathcal{F},:)$ \Comment{New post-failure values}
    \State \Return $Y$
\EndFunction

\Function{Bop}{$E, T, X, F$}
    \State $Y \gets X$; \quad $Y(F,:) \gets 0$
    \State $Y \gets Y + E\, X(F,:)$
    \State $Y(F,:) \gets Y(F,:) + E^{\top} X$
    \State $Y(F,:) \gets Y(F,:) + (T - E(F,:) - E(F,:)^{\top})\, X(F,:)$
    \State \Return $Y$
\EndFunction

\Function{invBop}{$G, E, EF, X, F$}
    \State $Y \gets X$; \quad $Y(F,:) \gets 0$
    \State $W \gets E^{\top} X$; \quad $V \gets X(F,:) + EF^{\top} X(F,:)$
    \State $Y \gets Y + E\, (G^{-1} W) - E\, (G^{-1} V)$
    \State $Y(F,:) \gets Y(F,:) + (I_F + EF)\, (G^{-1} (V - W))$
    \State \Return $Y$
\EndFunction
\end{algorithmic}
\end{algorithm}

The power method for solving the generalized eigenvalue problem for matrices $A$ and $B$ involves multiplying both on the left and right by an orthogonal matrix $Q$, as outlined in Algorithm~\ref{algo:gen_block_power_qr}. At each iteration, a QR factorization is performed to compute the orthonormal matrix $Q$. Similar to the standard power method, the sequence of vectors $X$ is updated by multiplying the orthogonal matrix $Q$ with the approximate eigenvectors $U$ obtained from solving the projected eigenproblem defined by the matrices $A_Q = Q^\top A Q$ and $B_Q = Q^\top B Q$. Convergence is assessed by computing the relative residual $r_{\text{rel}}$ and comparing it against a predefined tolerance $\varepsilon$ for the eigenvector corresponding to the dominant eigenvalue of the matrix pencil $(A, B)$.

In the generalized power method, the key step is an iteration is computing $B^{-1} A X$. The structure of our matrix $B$ makes this easy in the event of any fault. Recall that 
\[ B = \begin{bmatrix} I & E_C \\ E_C^T & E^T E \end{bmatrix}. \]
This means it is -- at most -- a rank $k$ update to the identity matrix $B$. Consequently, the matrix has a simple closed form inverse that can be used to implement an efficient solve 
\[ B^{-1} = \begin{bmatrix} 
I + E_C ( E^T E - E_C^TE_C)^{-1} E_C^T & -E ( E^T E - E_C^TE_C)^{-1} \\ 
- ( E^T E - E_C^TE_C)^{-1} E^T & ( E^T E - E_C^TE_C)^{-1}
\end{bmatrix}.
\]
Since $E = \left[ \begin{smallmatrix} E_C \\ E_F \end{smallmatrix} \right]$, we have $E^T E - E_C^T E_C = E_F^T E_F$. This matrix is always invertible because $E$  has Kruskal row rank $k$. 
Also, the inverse here only occurs for a small $k \times k$ block, $( E^T E - E_C^TE_C)^{-1}$, which is repeated through all the factors and could be treated practically with either a small inverse or a small LU factorization.  

The result is a set of operations that implement these operations without explicitly reforming the matrix $A$. The auxiliary procedures \texttt{Aop}, \texttt{Bop}, and \texttt{invBop}, which are invoked in Algorithm~\ref{algo:gen_block_power_qr_bsolve}, are specified in Algorithm~\ref{alg:ec-ops}. The key step is compute a Cholesky factor with all the solves. 


\subsection{Implemented An Erasure-Coded TraceMin}
Our erasure-coded eigensolver builds upon the TraceMin algorithm described in \cite{2000sameh}, and is detailed in Algorithm \ref{algo:tracemin}, which outlines the procedure for handling faults.
  
The eigenvectors $X$ of the original eigenvalue problem are recovered from the eigenvectors $X'$ of the reconstituted system, computed as $X$ in Algorithm~\ref{algo:tracemin}. 

\begin{algorithm}[tb]
\caption{Erasure-coded TraceMin}\label{algo:tracemin}
\textbf{Input:} $A$ symmetric matrix, subspace dimension $s > 0$, number of erasures tolerated $k > s$ \\
\textbf{Output:} $\Theta (eigenvalues), X (eigenvectors)$
\begin{algorithmic}[1]
\State{Generate the sparse coding matrix $E$ which can tolerate up to $k$ row/ column failures using coding matrix.}
\State{Compute augmented blocks ($Z$, $R$) and ($Q$, $S$) for matrix A and B respectively.}
\State{Choose a block size $s_2 = 2\times s$ and an $n \times s_2$ random matrix $V_1$ of full rank such that $V_{1}^{T}BV_{1} = I$.}
\For{$i = 1, 2, ...$ until convergence}
\If{Fault Occurs}
    \If{Fault Occurs at a Node (say F)}
    \State{Replace the faulty rows and columns of $A$ a with the $Z$ and $Z'$ blocks and the intersection of rows and columns with block $R$}.

    \State{Replace the same rows and columns of $B$ a with the $Q$ and $Q'$ blocks and the intersection of rows and columns with block $S$}.
    
    \State{Replacing the faulty rows and columns  with coded-blocks to reconstitute $A'$ and $B'$; $A = A'$; $B = B'$ }
\EndIf
\EndIf
\State Compute $W_{i} = AV_{i}$ and the interaction matrix $H_{i} = V_{i}^{T} W_{i}$.
\State Compute the eigenpairs of $(Y_{i};\Theta_{i})$ of  $H_{i}$. Sort eigenvalues in ascending order and rearrange the corresponding eigenvectors.
\State Compute the Ritz Vectors $X_{i} = V_{i}Y_{i}$.
\State Compute the residuals $R_{i} = AX_{i} - B X_{i}\Theta_{i}$.
\State Test for Convergence.
\State Solve the following linear system approximately via the CG.\\ $AZ_{i+1}=BX_{i}$
\State B-orthonormalize $Z_{i+1}$ into $V_{i+1}$.
\EndFor
\end{algorithmic}
\end{algorithm}

The TraceMin algorithm computes a few of the smallest eigenvalues by reducing the generalized
Rayleigh quotient ($\frac{X^{T}AX}{X^{T}BX}$) step by step. In \cite{1982Sameh}, a simultaneous iteration method was introduced to address this problem. In each iteration, the previous approximation $X_{k}$ , which satisfies $X_{k}^{T}BX_{k} = I_{s}$ and $X_{k}^{T}AX_{k}= \Theta_{k}$, is updated with a correction term $\Delta_{k}$, calculated as follows:
\begin{equation}\label{eq:traceminimization}
\begin{aligned}
    min \quad & tr(X_{k}-\Delta_{k})^{T}A(X_{k}-\Delta_{k}), \\ 
    subject \: to \quad &X_{k}^{T}B\Delta_{k} = 0
\end{aligned}
\end{equation}
for any B-orthonormal basis $X_{k+1}$ of the subspace spanned by ${Z_{k+1}}$, we construct our erasure-coded scheme on top of the TraceMin algorithm, as described in Algorithm 2 of \cite{2000sameh}. In Step 1 of Algorithm~\ref{algo:tracemin}, we augment the matrices $A$ and $B$ with erasure-coded rows and columns. In the event of a fault, the system is reconstituted from the erasure coding blocks, as outlined in Step 5 of the Algorithm~\ref{algo:tracemin}. The remainder of the implementation follows Algorithm 2 in \cite{2000sameh}.

\subsection{Construction of Erasure Coding Matrix}\label{sec:coding_struc}

Although in theory a Kruskal rank of $k$ is required to guarantee recovery from every possible fault pattern, a weaker but more practical requirement significantly reduces computational complexity. This approach, referred to as \emph{recovery at random} in \cite{joyceDynamic}, assumes that faults occur randomly and can be recovered with exceptionally high probability. Furthermore, this recovery method leverages a structured sparse matrix, introducing minimal computational overhead while maintaining high efficiency.

\subsection{Coding Matrix Structure}
The coding matrix must be sparse to optimize both computation and storage. To achieve this, we generate a random matrix of dimensions $n \times p$ and arrange it in a staggered pattern to preserve the sparsity of the erasure coding block. Specifically, we ensure that each row of size $k$ contains $p$ nonzero entries, as proposed in ~\cite{joyceDynamic} (Figure \ref{fig:codingstructure}).

More specifically, the coding matrix $E$ from~\cite{joyceDynamic} is an $n\times k$ matrix constructed using a staggered nonzero pattern with $p$ randomly chosen entries for each occurrence of this pattern.\cite{joyceDynamic} showed that selecting $p$ to be larger than $\frac{\log{k}}{\log{\log{k}}}$ ensures that the probability of a random set of $k$ rows of matrix $E$ drawn from the $n \times k$ p-staggered distribution being linearly dependent is less than $\big(\frac{e}{p+1}\big)^{p+1}$. Furthermore, they demonstrated that, with high probability, the maximum number of rows (out of $k$ randomly chosen rows) in matrix $E$ that share the same nonzero structure is approximately $\frac{\ln{k}}{\ln{\ln{k}}}(1 + o(1))$.  

\begin{wrapfigure}{r}{0.5\textwidth}
  \begin{center}
    \includegraphics[width=0.2\linewidth]{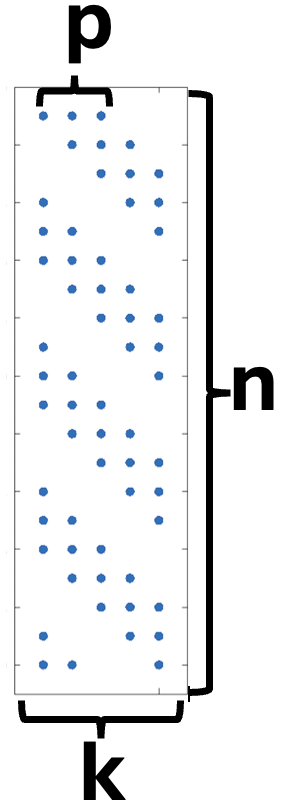}
  \end{center}
  \vspace{-0.5em}
    \caption{Coding Matrix}
    \vspace{-1.5em}
    \label{fig:codingstructure}
\end{wrapfigure}   We illustrate this $n\times k$ coding matrix, using the staggered non-zero pattern with $p$ non-zeros entries, in Figure \ref{fig:codingstructure}. 
We adapt this coding matrix to improve its design for eigenvalues and eigenvectors, which are more sensitive than linear systems. The erasure of critical row-column pairs can significantly affect the stability and convergence of the eigensolver. In the case of adversarial deletions, the rows that contribute most to the eigenvalues may be erased. Therefore, assigning more weight to the rows based on their importance is crucial to preserve the integrity of the data. We used leverage scores to assign appropriate importance to specific rows. In addition to random coding schemes, we also evaluated Leverage Score-weighted Coding, as described below.

\section{Experimental Results}
\label{sec:experiments}
We present detailed experimental validation of our erasure-coded fault-tolerant TraceMin eigensolver on both dense and sparse benchmark datasets described in Table \ref{tbl:benchmark}. The goal of these experiments is to demonstrate the the stability and robustness of our approach. All methods and plots are implemented in MATLAB on an Apple M1 Pro with 8 cores with 16 GB RAM.

\begin{table}[H]
\begin{tabular}{ p{0.9cm}p{2.7cm}p{1.7cm}p{1.2cm}p{6.7cm}  }
 \toprule 
  Type & Dataset & Dims. & Nonzeros & Description\\
 \midrule 
 \multirow{4}{*}{Dense} & MNIST - Train & $15K \times 15K$ & $225M$ & Training Dataset \\
 &MNIST - Test & $10K \times 10K$ & $100M$ & Test Dataset \\
&CIFAR-10 - Train  & $15K \times 15K$ & $225M$ & Training Dataset\\
&CIFAR-10 - Test & $10K \times 10K$ & $225M$  & Test Dataset \\
\midrule 
\multirow{8}{*}{Sparse}&bcsstk17 & $11K \times 11K$ & $429K$ &  Stiffness Matrix - Elevated Pressure Vessel \\ 
&bcsstk25 & $15K \times 15K$ & $252K$ & Stiffness Matrix - 76 Story Skyscraper \\
&gyro & $17K \times 17K$ & $1M$  & Model Reduction Problem \\
&msc23052 & $23K \times 23K$ & $1.15M$  & MSC/NASTRAN Shock Problem \\
&cbuckle & $14K \times 14K$ & $677K$  & Cylindrical Shell Stiffness Matrix\\
&Pres\_Poisson & $15K \times 15K$ & $716K$  & Computational Fluid Dynamics Problem \\
&jnlbrng1 & $40K \times 40K$ & $199K$  & Quadratic Journal Bearing Problem \\
&torsion1 & $40K \times 40K$ & $198K$  & Optimization Problem\\
\bottomrule 
\end{tabular}
\captionof{table}{Benchmark Datasets}
\label{tbl:benchmark}
\vspace{-10pt}
\end{table}


\subsection{Power Method vs TraceMin}
We compare the performance results of the Power Iteration Method and TraceMin across various sparse and dense datasets, as summarized in Table \ref{tbl:benchmark}. Two sets of experiments are conducted to analyze the performance impact: (1) as the problem size increases and (2) as the subspace size expands.

As the problem size increases, the runtime is expected to grow. However, it is crucial to compare the two methods to quantify their performance impact, enabling us to identify the most efficient approach for computing eigenvalues. Additionally, this comparison helps evaluate the overhead introduced by erasure coding relative to the best-performing method.

\begin{figure}[H]
\centering
\begin{subfigure}{.49\textwidth}
  \centering
  \includegraphics[width=.98\linewidth]{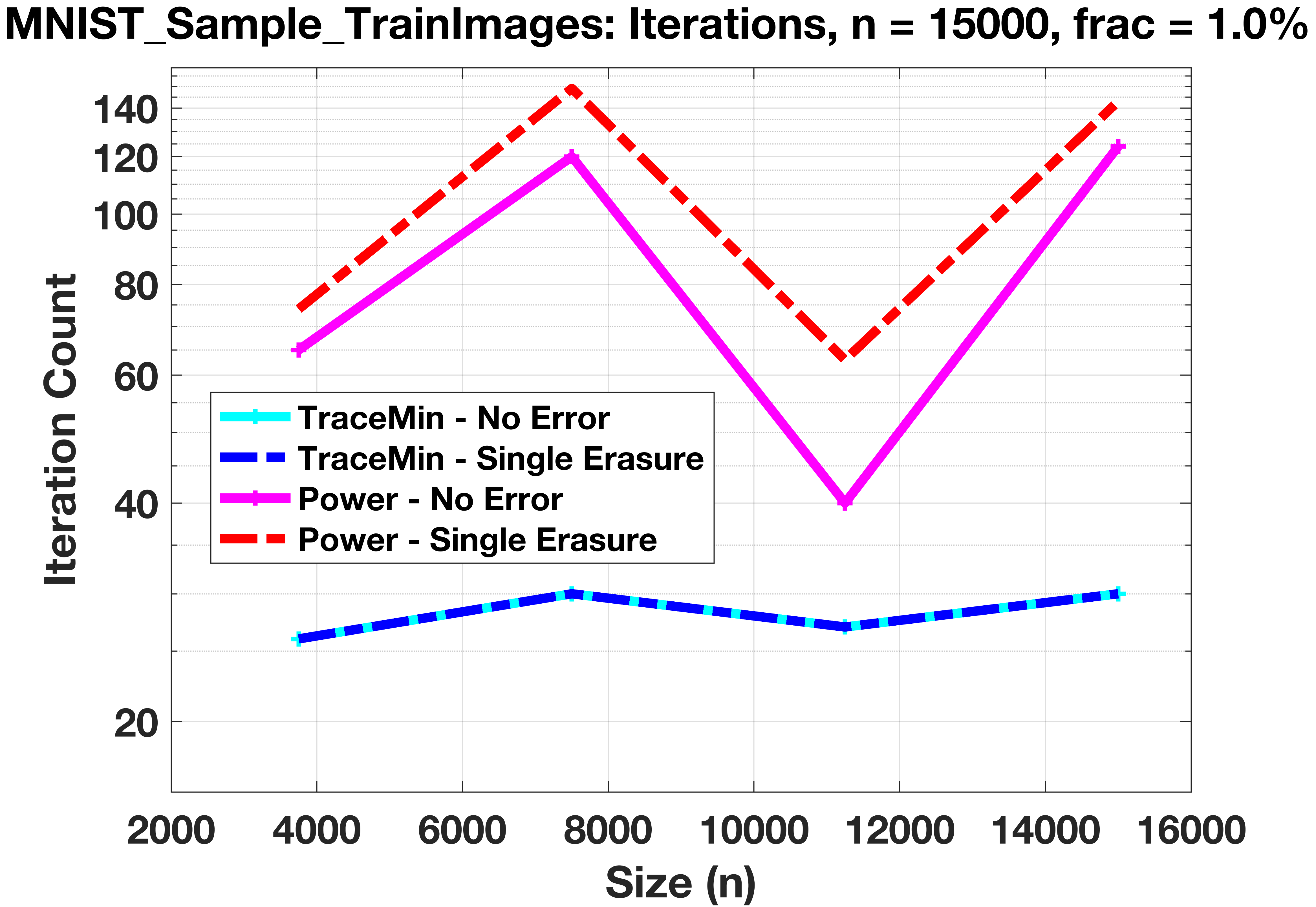}
  \caption{Iterations}
  \label{fig:power-size-iter-mnist}
\end{subfigure}%
\begin{subfigure}{.49\textwidth}
  \centering
  \includegraphics[width=.98\linewidth]{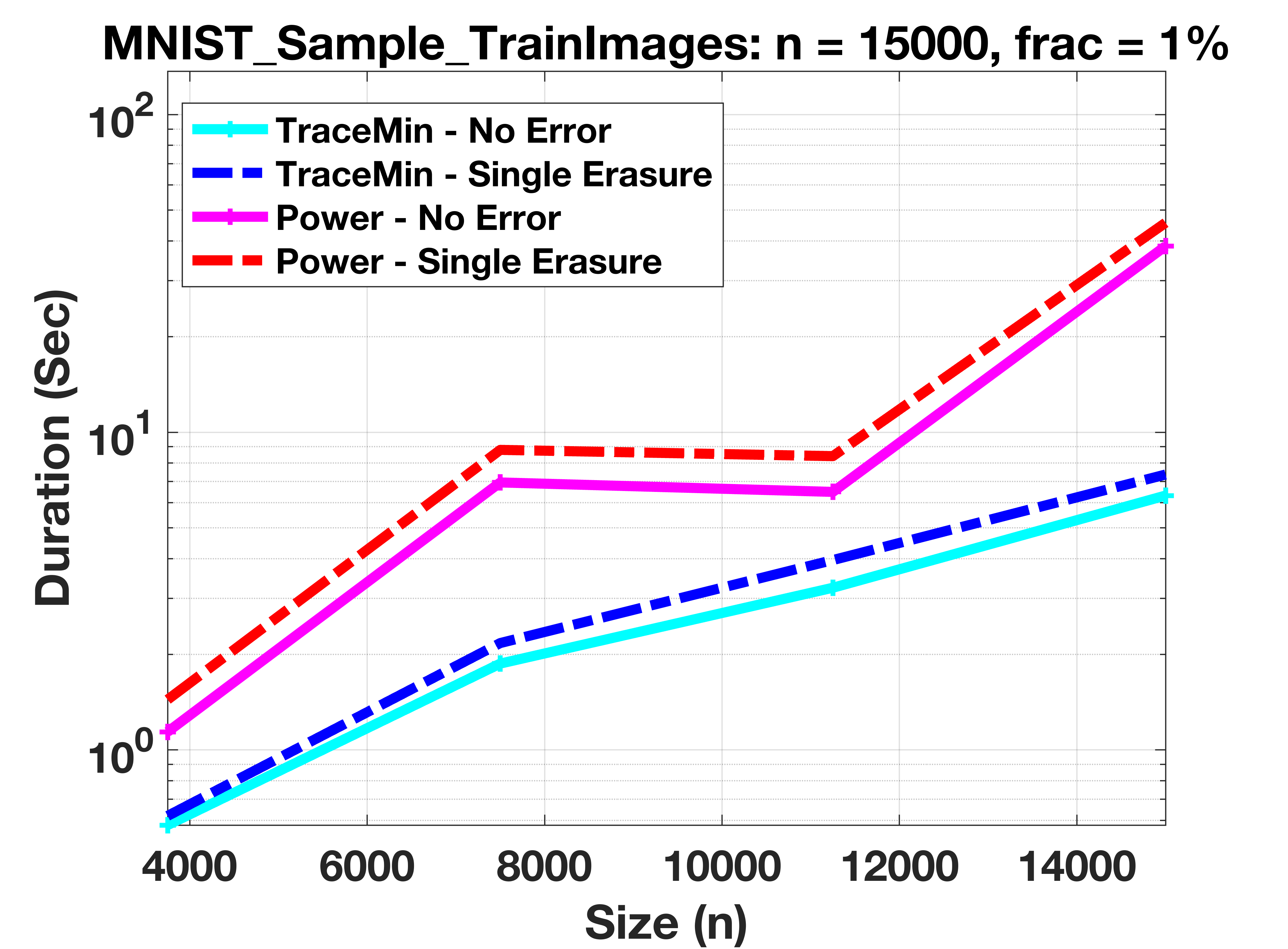}
  \caption{Timing}
  \label{fig:power-size-time-mnist}
\end{subfigure}
\caption{MNIST Train Dataset with 0.1\% Erased}
\label{fig:power-size-mnist}
\end{figure}

Figures~\ref{fig:power-size-iter-mnist} and~\ref{fig:power-size-time-mnist} compare the convergence behavior and computational cost of the \textbf{TraceMin} and \textbf{Power Method} algorithms under both \textbf{no-error} and \textbf{single-fault} scenarios for varying problem sizes on the MNIST dataset.

Figure~\ref{fig:power-size-time-mnist} clearly demonstrates that while erasure coding introduces some overhead for both the Power Method and TraceMin, the TraceMin algorithm consistently outperforms the Power Method, achieving significantly faster runtimes—often by an order of magnitude—across all matrix sizes. Although the runtime of both methods increases with problem size, the performance degradation is far more pronounced for the Power Method. The Single Erasure cases for both algorithms exhibit higher runtimes than their no-error counterparts due to the additional fill-in introduced by erasure coding—nonzero coding rows and columns that increase the effective matrix density and, consequently, the computational cost per iteration. Overall, TraceMin demonstrates superior scalability and robustness to erasures, whereas the Power Method shows a steeper increase in both runtime and iteration count, underscoring its inefficiency for large-scale or fault-tolerant eigenvalue computations.

\begin{figure}[H]
\centering
\begin{subfigure}{.49\textwidth}
  \centering
  \includegraphics[width=.98\linewidth]{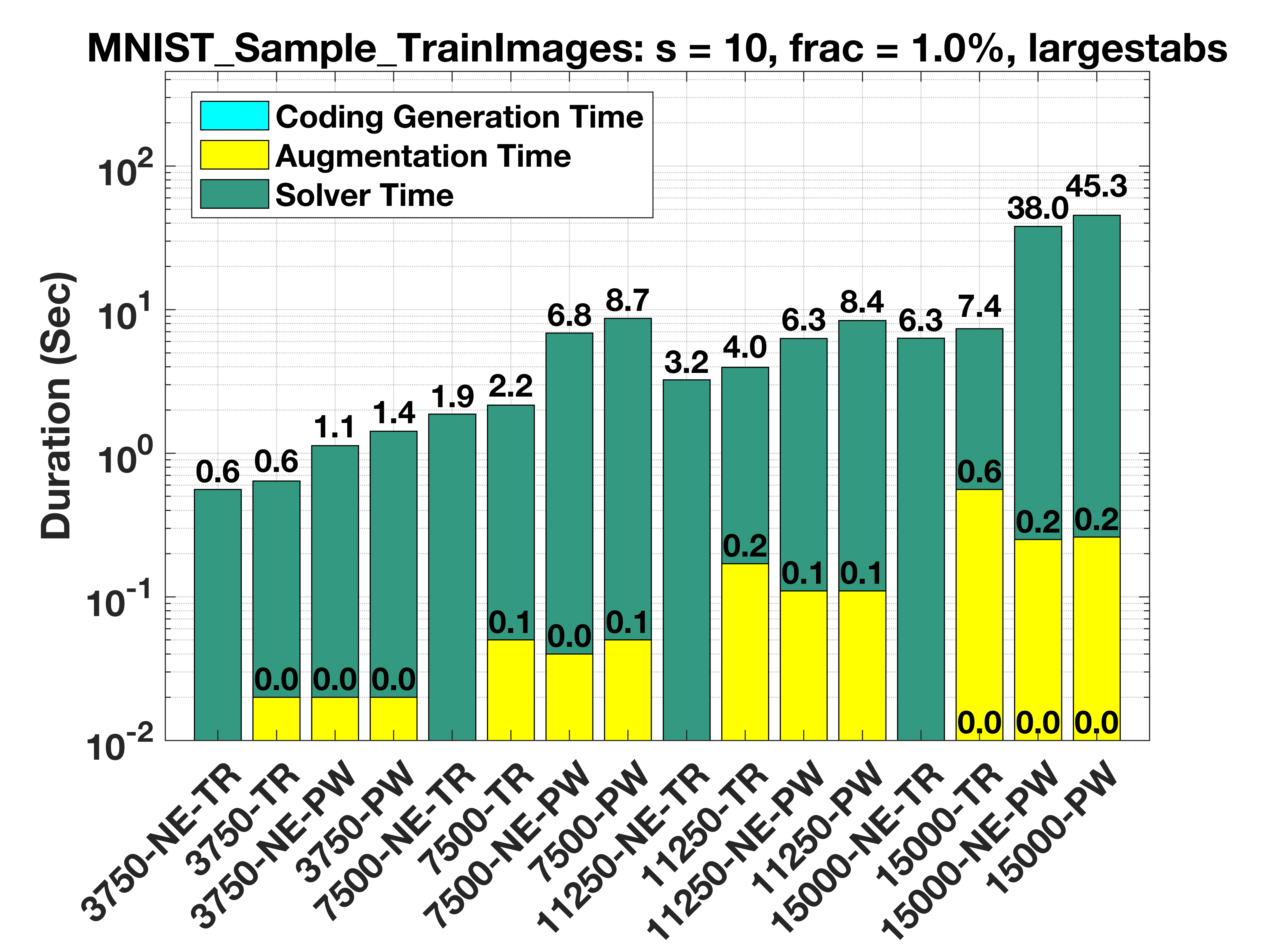}
  \caption{MNIST}
  \label{fig:power-size-time-breakdown-mnist}
\end{subfigure}%
\begin{subfigure}{.49\textwidth}
  \centering
  \includegraphics[width=.98\linewidth]{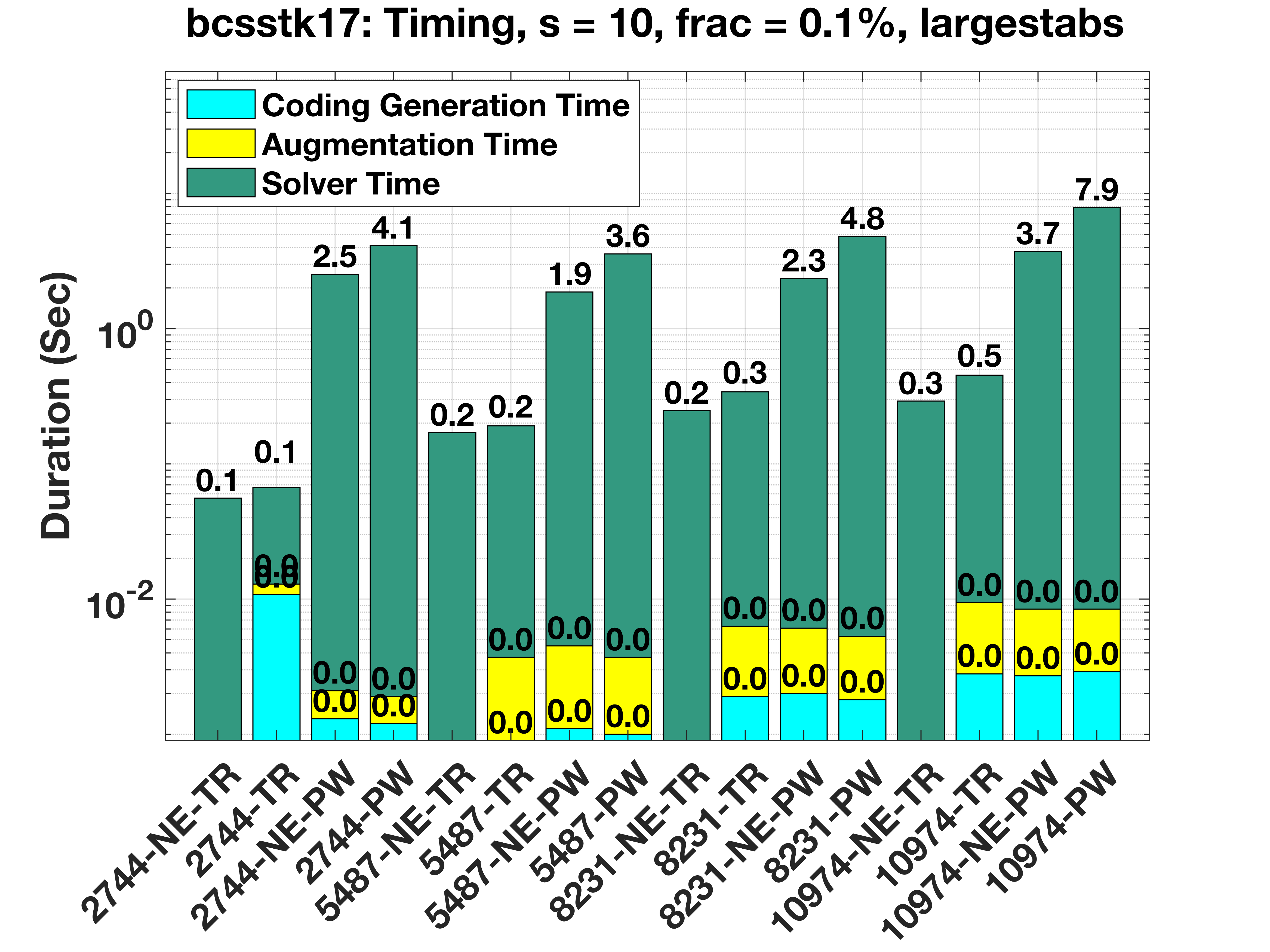}
  \caption{bcsstk17}
  \label{fig:power-size-time-breakdown-bcsstk17}
\end{subfigure}
\caption{Timing Breakdown}
\label{fig:time-breakdown}
\end{figure}
In Figure \ref{fig:time-breakdown}, the runs for Power Iteration with optimized solver (Algorithm \ref{algo:gen_block_power_qr_bsolve}) and TraceMin (Algorithm \ref{algo:tracemin}) are labeled with the suffixes ``PW'' and ``TR'', respectively, while ``NE'' represents the ``No Erasure'' scenario. The figure illustrates that in both dense matrix scenarios (e.g., MNIST) and sparse matrix scenarios (e.g., bcsstk17), TraceMin is significantly faster than the Power Iteration method. This disparity arises because the solver time is substantially higher for the Power Iteration method compared to TraceMin.

\begin{figure}[H]
\centering
\begin{subfigure}{.49\textwidth}
  \centering
  \includegraphics[width=.98\linewidth]{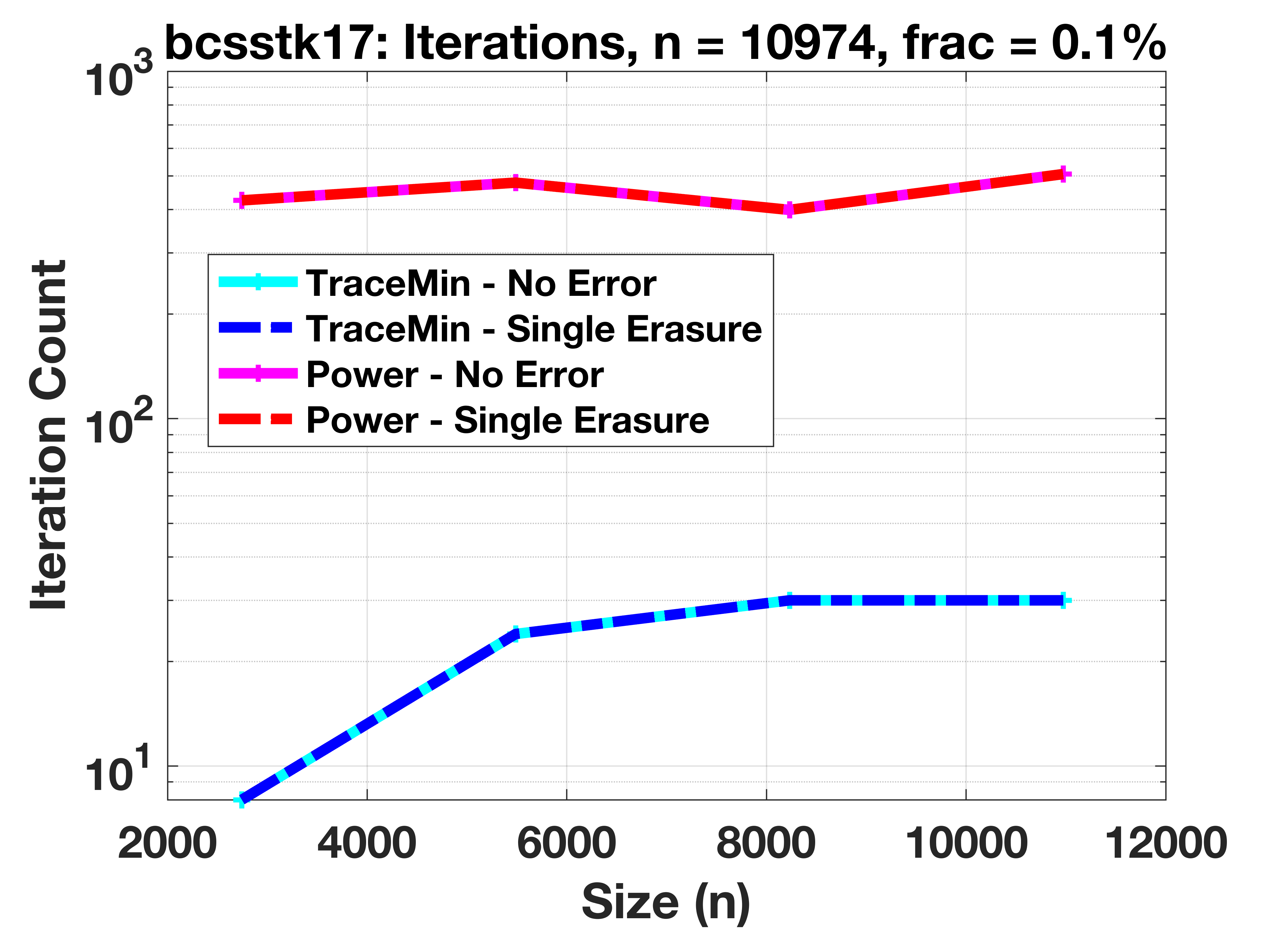}
  \caption{Iterations}
  \label{fig:power-size-iter-bcsstk17}
\end{subfigure}%
\begin{subfigure}{.49\textwidth}
  \centering
  \includegraphics[width=.98\linewidth]{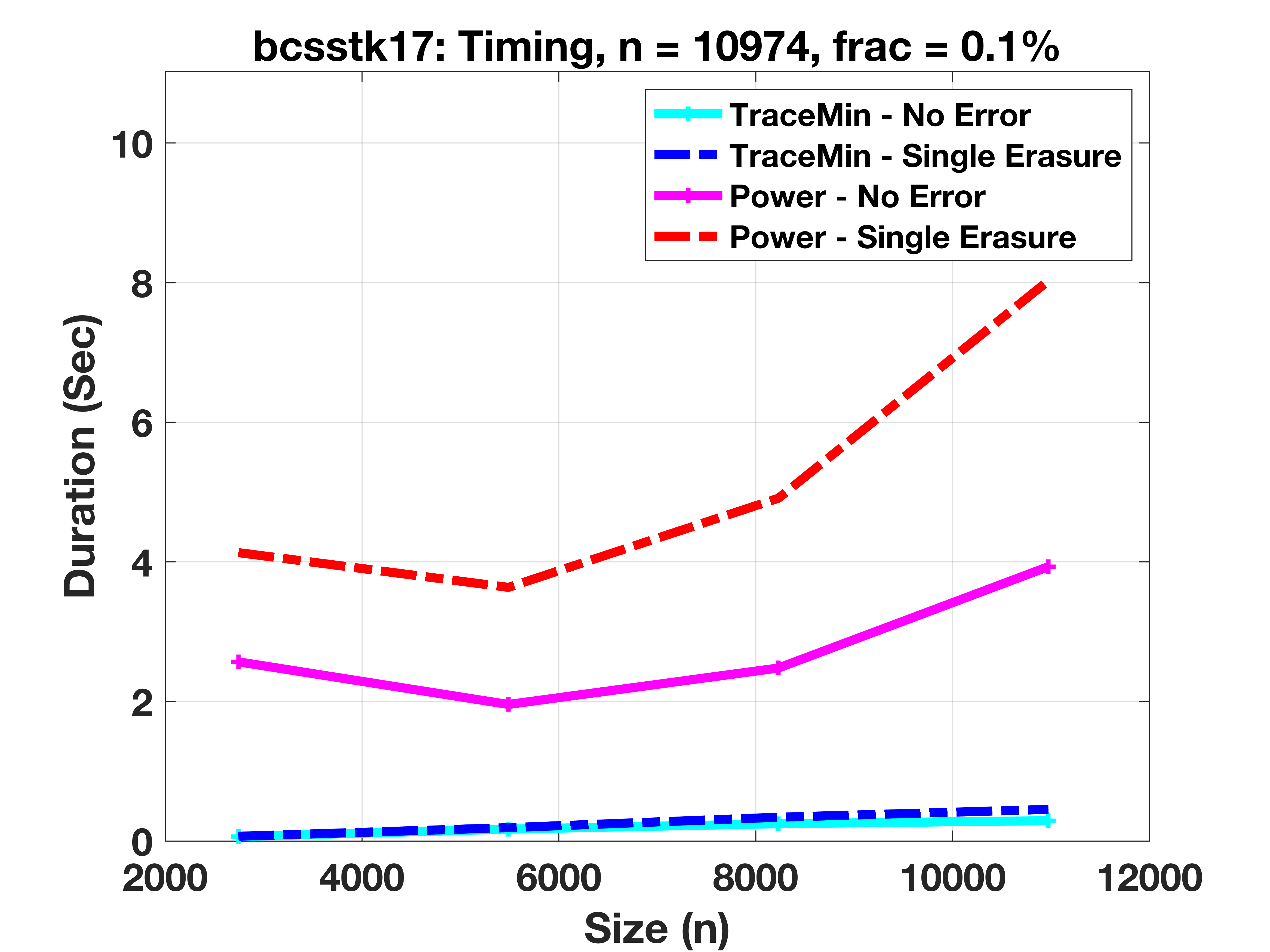}
  \caption{Timing}
  \label{fig:power-size-time-bcsstk17}
\end{subfigure}
\caption{Sparse bcsstk17 Train Dataset with 0.1\% Erased}
\label{fig:power-size-bcsstk17}
\end{figure}

As shown in Figure~\ref{fig:power-size-iter-bcsstk17}, the Power Method (Algorithms~\ref{algo:gen_block_power_qr} and \ref{algo:gen_block_power_qr_bsolve}) requires significantly more iterations to converge for sparse dataset \textbf{bcsstk17} compared to Erasure-Coded TraceMin (Algorithm~\ref{algo:tracemin}), leading to a higher overall execution time, as illustrated in Figure~\ref{fig:power-size-time-bcsstk17}. Furthermore, Figure~\ref{fig:power-size-time-bcsstk17} demonstrates that the erasure scenario incurs a higher runtime than its no-erasure counterpart. This increase in computational cost arises because erasure coding introduces additional fill-in in the sparse system due to nonzero coding rows and columns, thereby increasing the time required for each iteration and ultimately resulting in a longer end-to-end execution time. The results clearly indicate that the TraceMin algorithm substantially outperforms the Power Method across all problem sizes. The runtime of TraceMin remains consistently low and scales modestly with increasing matrix size, demonstrating its computational efficiency and scalability. In contrast, the Power Method exhibits significantly higher runtimes, particularly under the Single Erasure scenario, where the cost increases sharply with problem size.

 The Power Method is designed to compute the largest eigenvalue and its corresponding eigenvector by using subspace iterations. As the subspace size increases, the size of eigenvectors increases as well and hence, the computational requirements grow significantly, making the method progressively more resource-intensive and costly to execute.

\begin{figure}[H]
\centering
\begin{subfigure}{.49\textwidth}
  \centering
  \includegraphics[width=.98\linewidth]{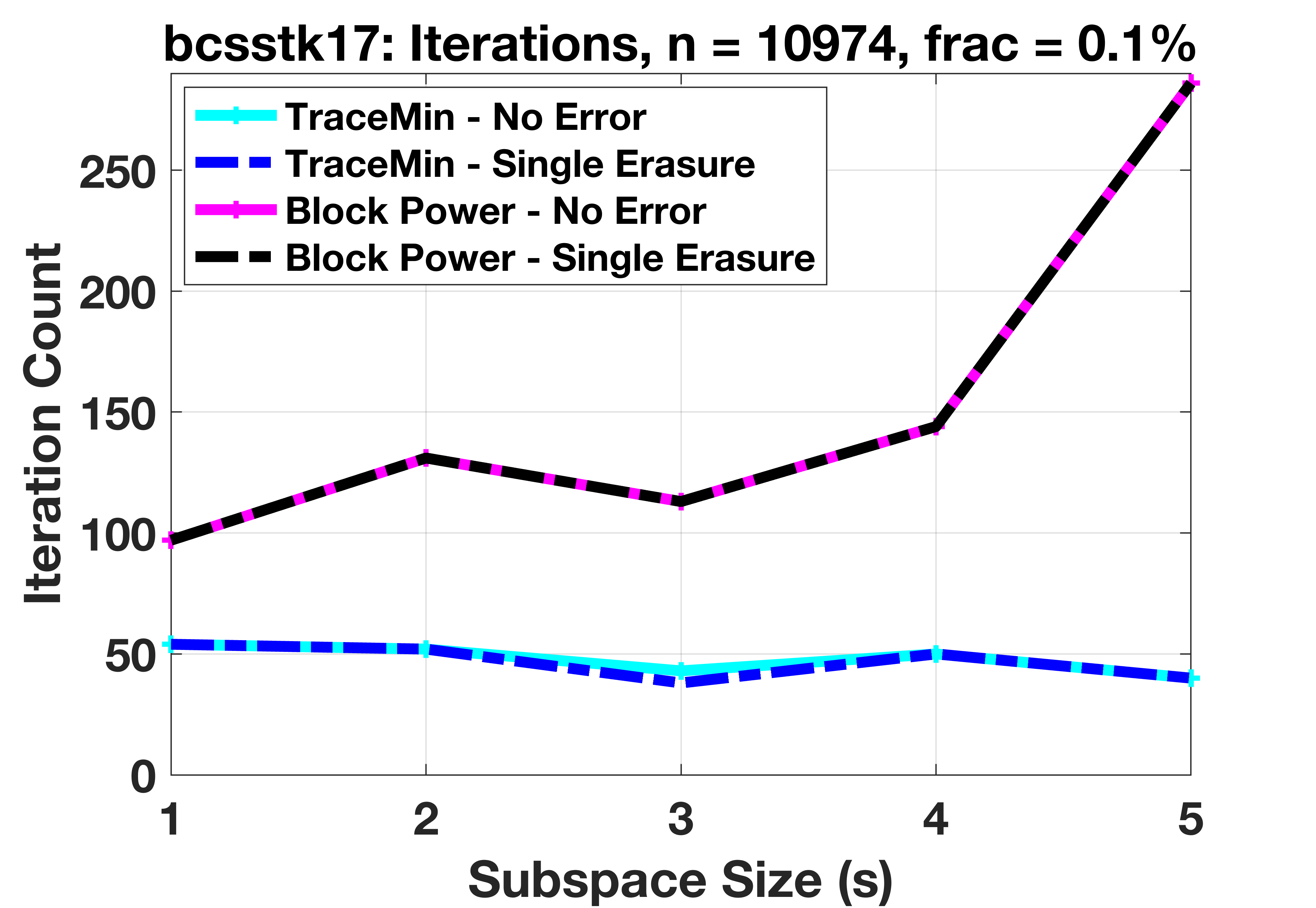}
  \caption{Iterations}
  \label{fig:power-s-iter-bcsstk17}
\end{subfigure}%
\begin{subfigure}{.49\textwidth}
  \centering
  \includegraphics[width=.98\linewidth]{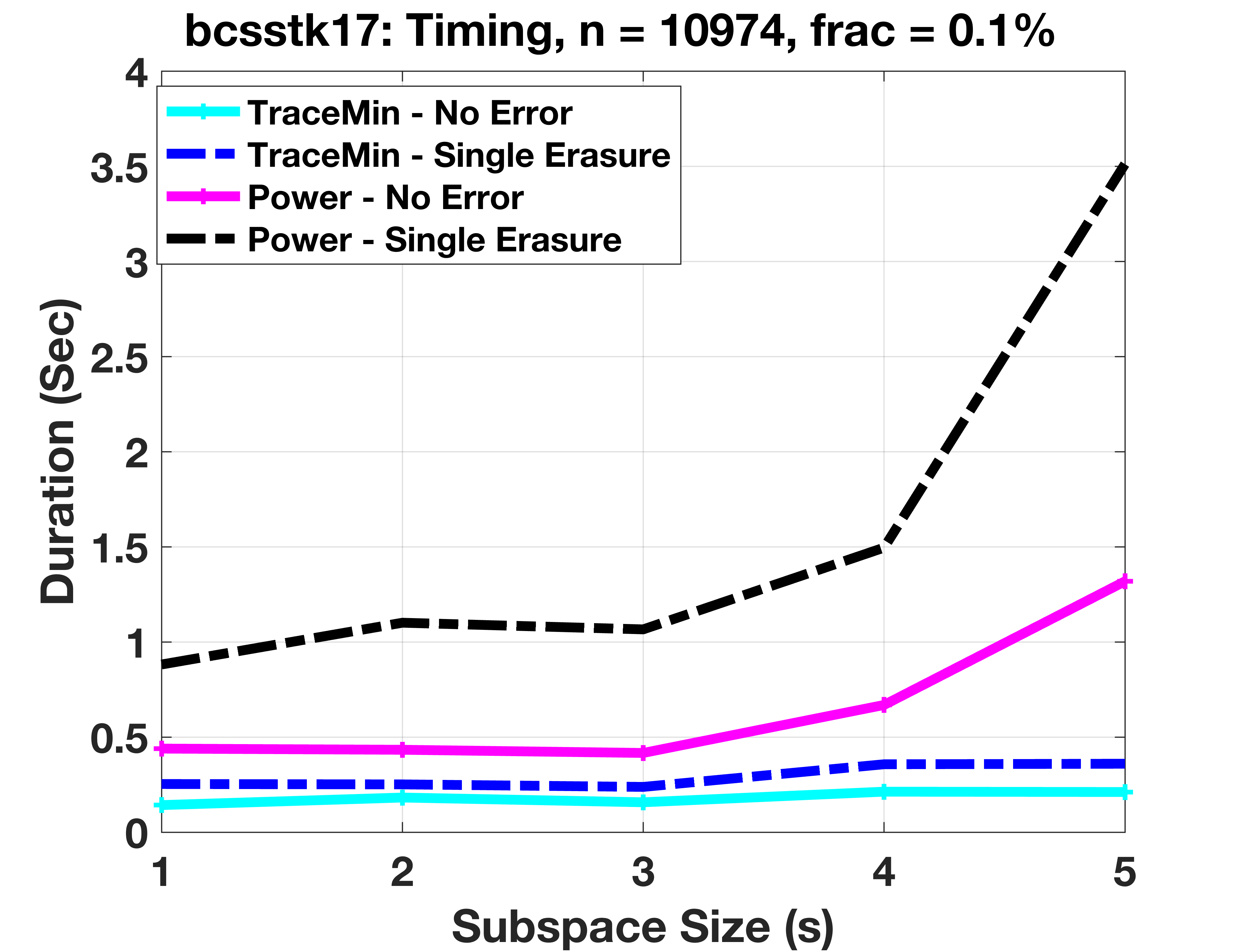}
  \caption{Timing}
  \label{fig:power-s-time-bcsstk17}
\end{subfigure}
\caption{Convergence for bcsstk17 Dataset for varying subspace size with 0.1\% Erased}
\label{fig:power-subspace-bcsstk17}
\end{figure}
\color{black}
Figure~\ref{fig:power-subspace-bcsstk17} illustrates that as the subspace size increases, the runtime of TraceMin remains relatively stable, whereas the Power Method becomes progressively more expensive due to the sequential computation of eigenvalues. Furthermore, we observe that the erasure-coded scenarios incur slightly higher runtimes than their no-erasure counterparts, primarily because erasure coding introduces additional fill-in arising from the added coding rows and columns.

\begin{figure}[H]
\centering
\begin{subfigure}{.49\textwidth}
  \centering
  \includegraphics[width=.98\linewidth]{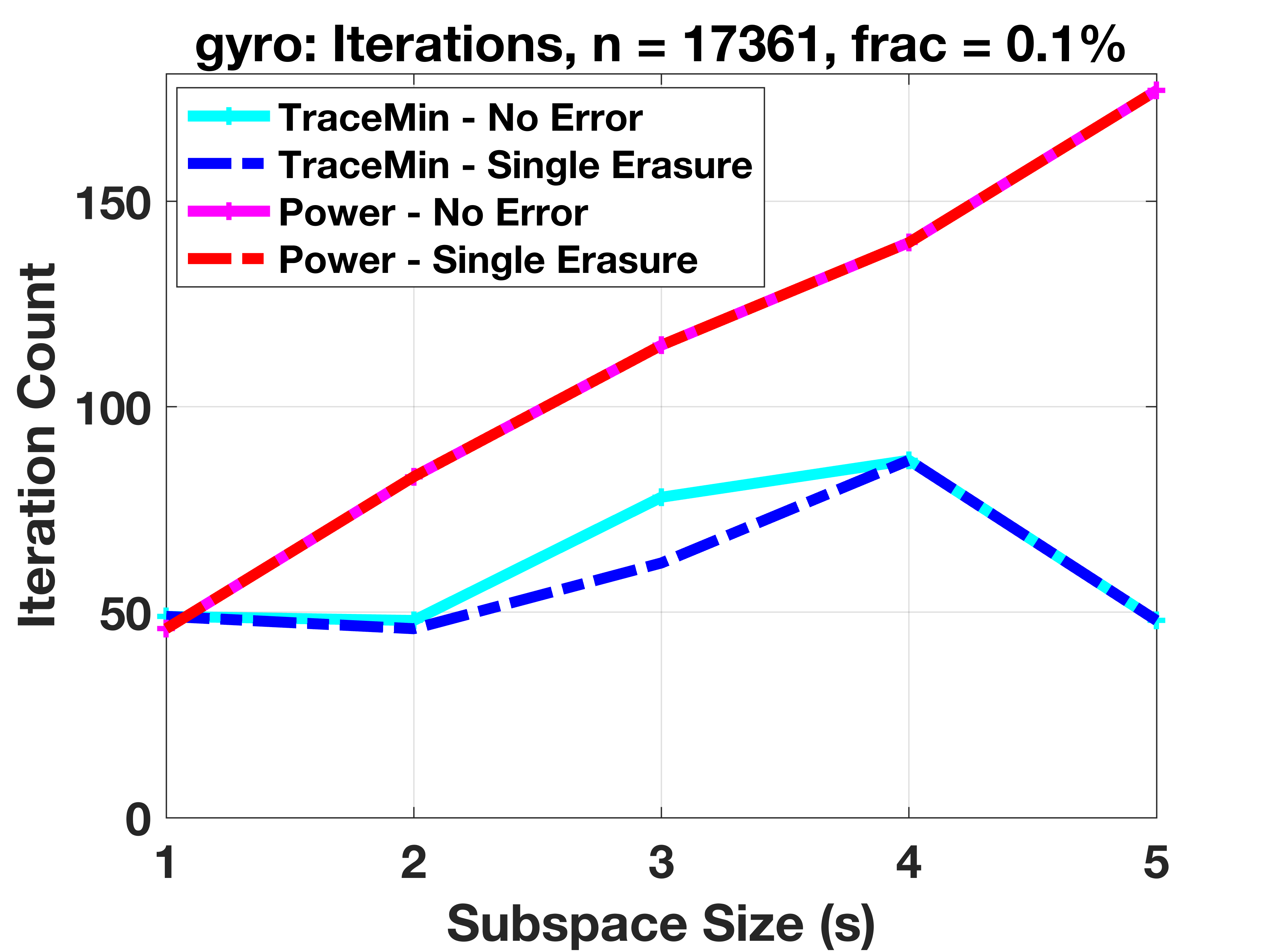}
  \caption{Iterations}
  \label{fig:power-s-iter-gyro}
\end{subfigure}%
\begin{subfigure}{.49\textwidth}
  \centering
  \includegraphics[width=.98\linewidth]{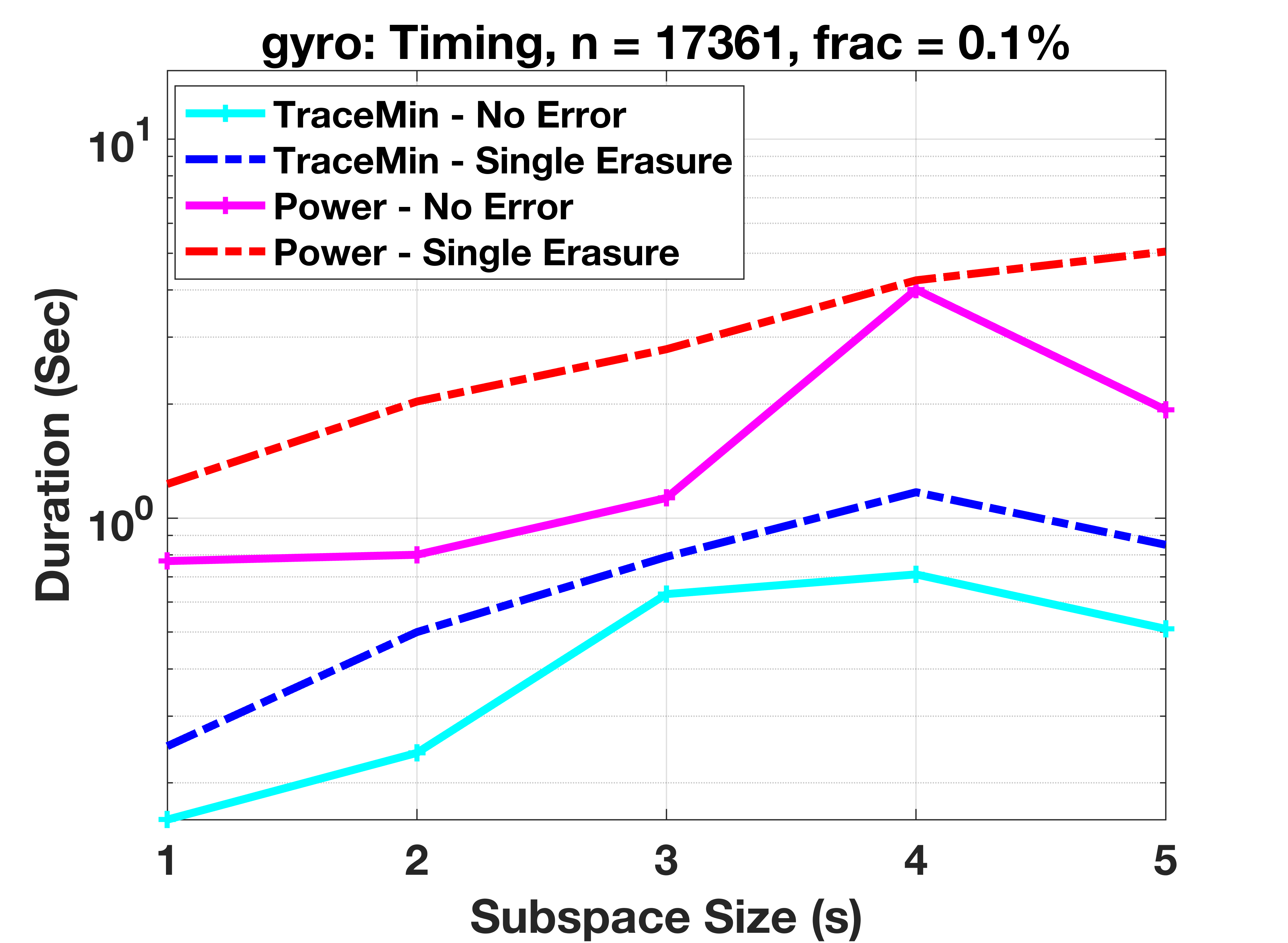}
  \caption{Timing}
  \label{fig:power-s-time-gyro}
\end{subfigure}
\caption{Convergence of Eigensolvers for gyro Dataset for varying subspace size}
\label{fig:power-subspace-gyro}
\end{figure}

Similarly, as shown in Figure~\ref{fig:power-s-iter-gyro}, the Power Method requires more iterations to converge on the \textbf{gyro} dataset when computing the dominant eigenvalues, resulting in significantly higher execution times compared to TraceMin (as illustrated in Figure~\ref{fig:power-s-time-gyro}). Since TraceMin consistently outperforms the Power Iteration method, the remainder of this paper focuses on analyzing the overheads introduced by erasure coding and comparing other relevant benchmark results with respect to TraceMin.

\subsection{Convergence Results}

We present here convergence results for covariance matrices computed from covariance matrix based on MNIST \cite{deng2012mnist} and CIFAR-10 \cite{Krizhevsky09learningmultiple} Train and Test datasets. 

\begin{figure}[H]
\centering
\begin{subfigure}{.49\textwidth}
  \centering
  \includegraphics[width=.98\linewidth]{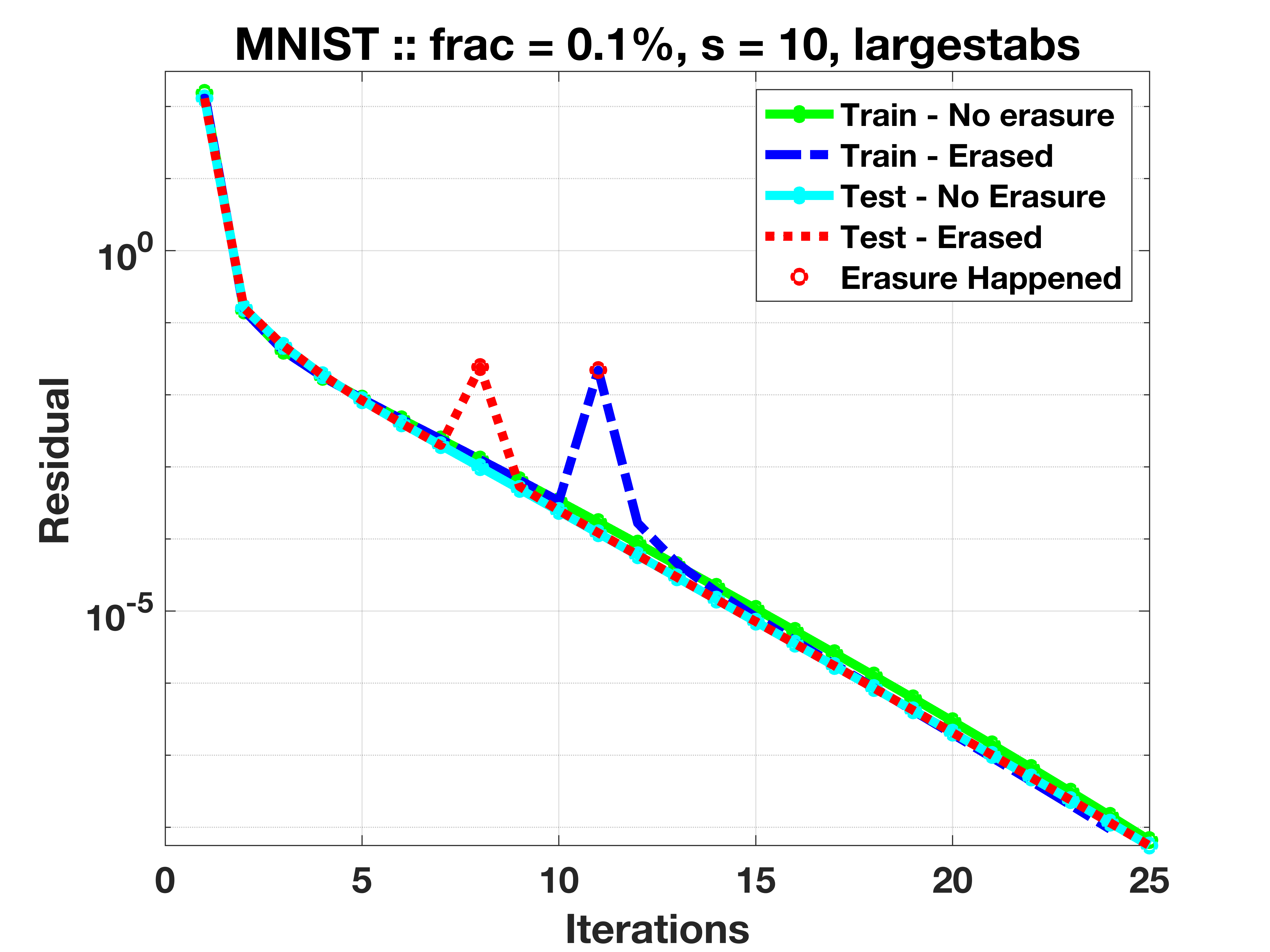}
  \caption{Residual}
  \label{fig:res-iter-mnist}
\end{subfigure}%
\begin{subfigure}{.49\textwidth}
  \centering
  \includegraphics[width=.98\linewidth]{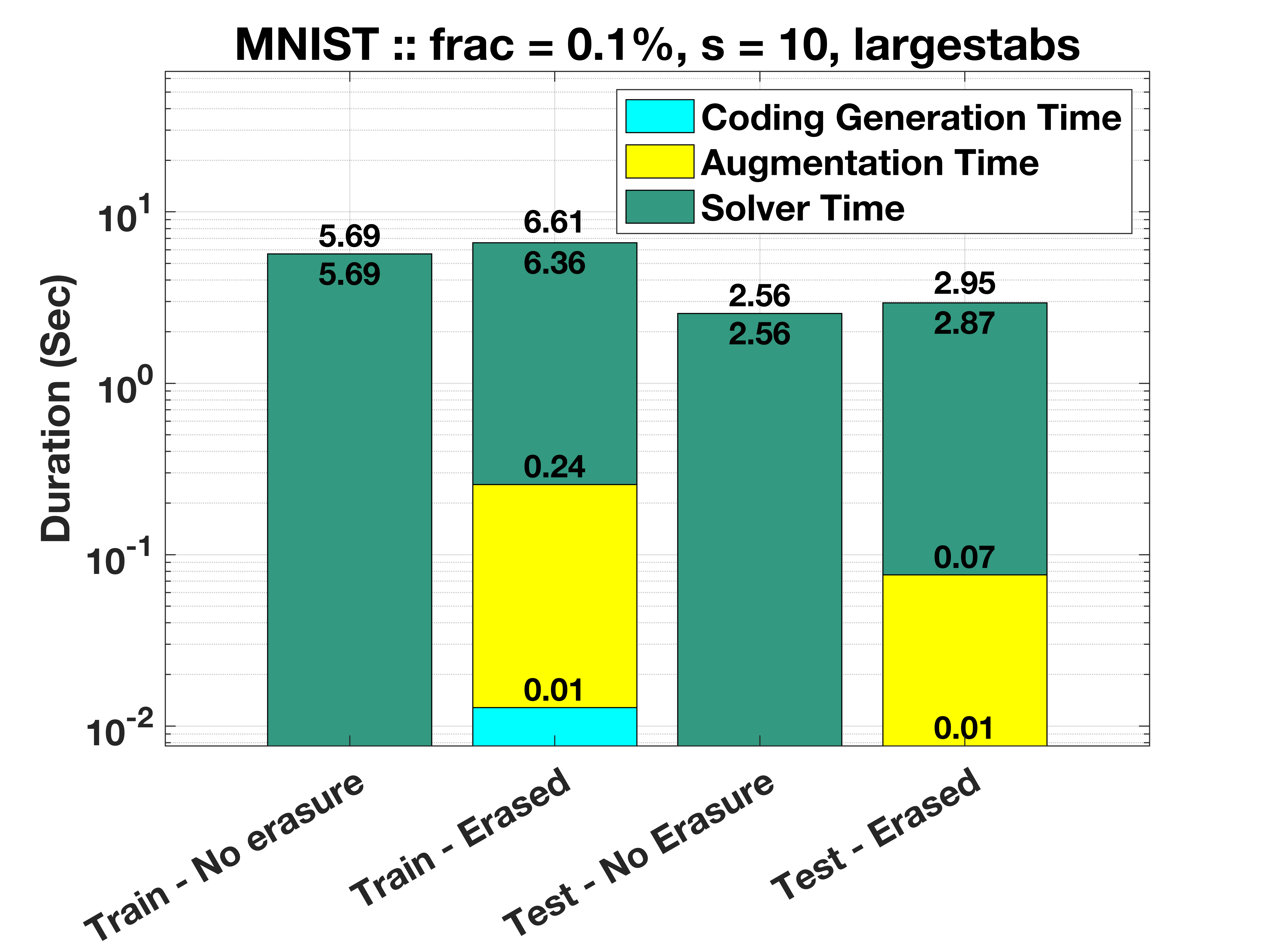}
  \caption{Timing Breakdown}
  \label{fig:res-time-mnist}
\end{subfigure}
\caption{Computing Largest Few Eigenvalue for MNIST Dataset with 0.1\% Erased}
\label{fig:res-mnist}
\end{figure}

Figure \ref{fig:res-iter-mnist} shows that the Erasure Coded TraceMin successfully recovers from faults and, within a few iterations, follows a residual pattern similar to its ``No Erasure'' counterparts. This indicates that Erasure-Coded TraceMin exhibits excellent convergence characteristics for MNIST Train and Test datasets. Furthermore, as shown in Figure \ref{fig:res-time-mnist}, Erasure Coding introduces minimal additional computational time compared to the ``No Erasure'' scenario, resulting in a significantly low overhead for the Erasure Coding scheme.

\begin{figure}[H]
\centering
\begin{subfigure}{.49\textwidth}
  \centering
  \includegraphics[width=.98\linewidth]{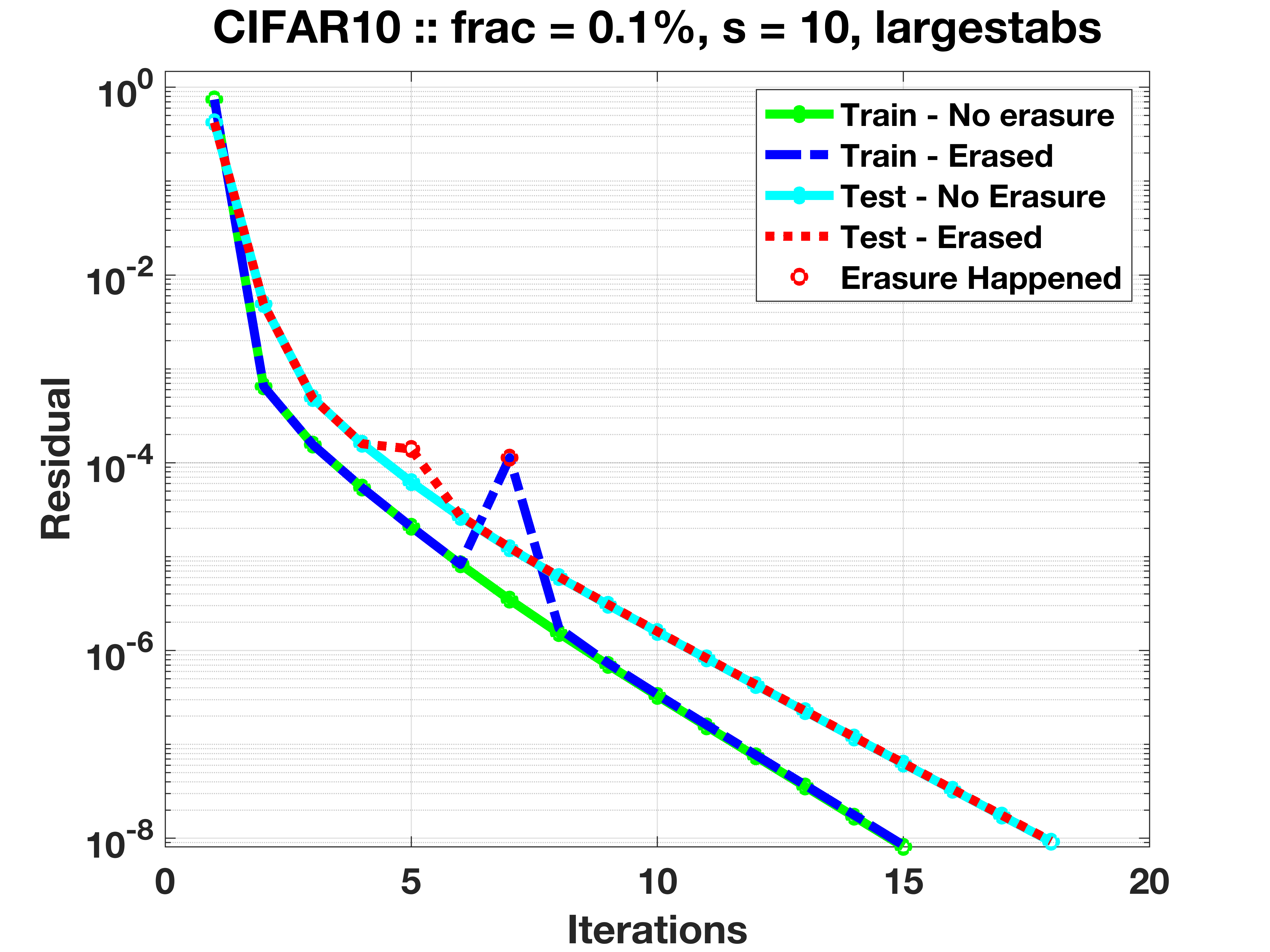}
  \caption{Residual}
  \label{fig:res-iter-cifar}
\end{subfigure}%
\begin{subfigure}{.49\textwidth}
  \centering
  \includegraphics[width=.98\linewidth]{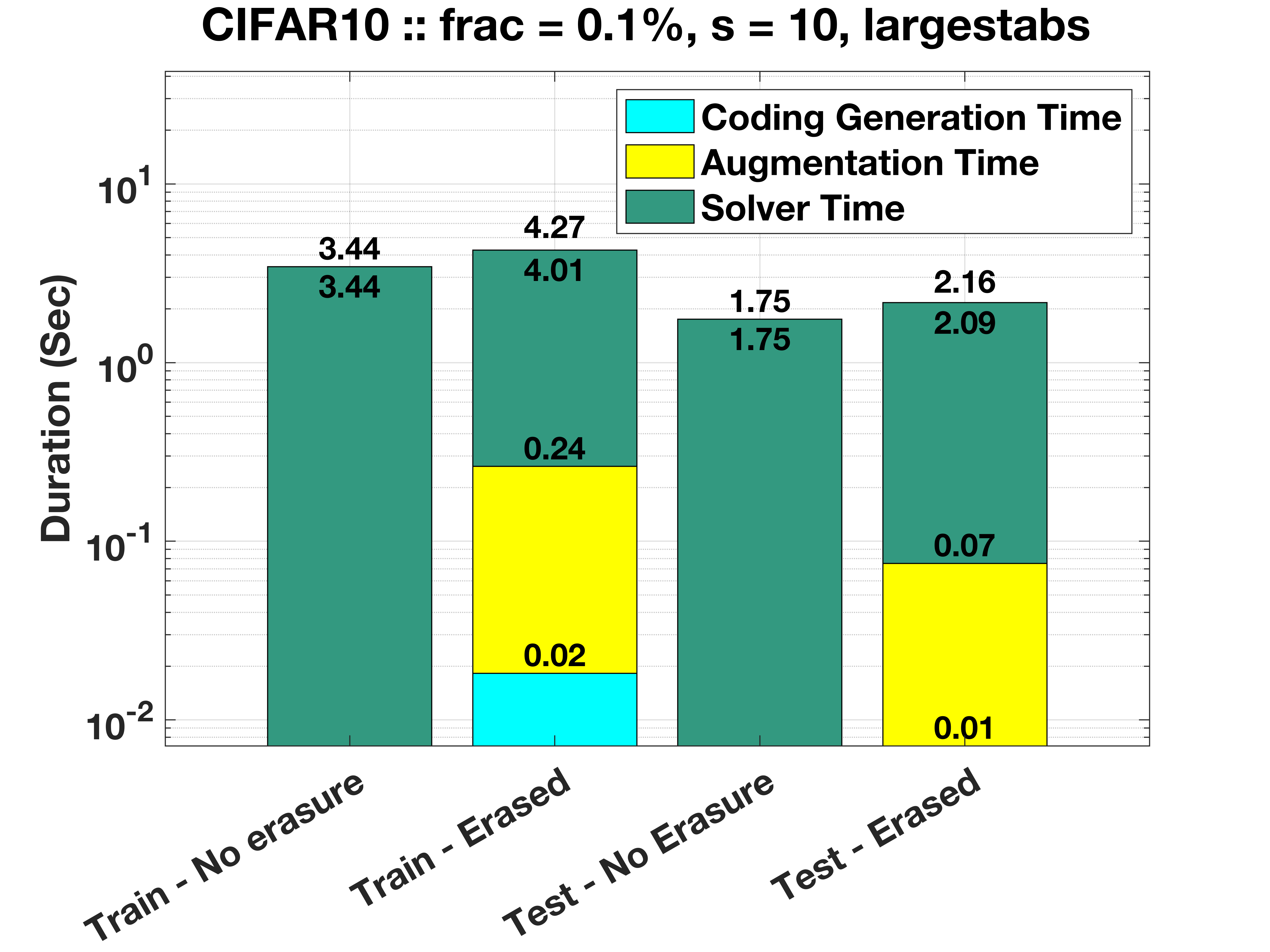}
  \caption{Timing Breakdown}
  \label{fig:res-time-cifar}
\end{subfigure}
\caption{Computing Largest Few Eigenvalue for CIFAR10 Dataset with 0.1\% Erased}
\label{fig:res-cifar}
\end{figure}

Figures \ref{fig:res-mnist} and \ref{fig:res-cifar} illustrate the TraceMin iteration errors for the MNIST and CIFAR-10 training and test datasets. We observe that while the residual increases in the event of an erasure, the Erasure-Coded TraceMin still converges efficiently without requiring significantly more iterations. This demonstrates the excellent convergence properties of the Erasure-Coded Eigensolver when computing the largest eigenvalues. Furthermore, Figure \ref{fig:res-time-cifar} highlights the minimal overhead introduced by Erasure-Coded TraceMin, making it a highly efficient and robust choice.

\begin{figure}[H]
\centering
\begin{subfigure}{.49\textwidth}
  \centering
  \includegraphics[width=.98\linewidth]{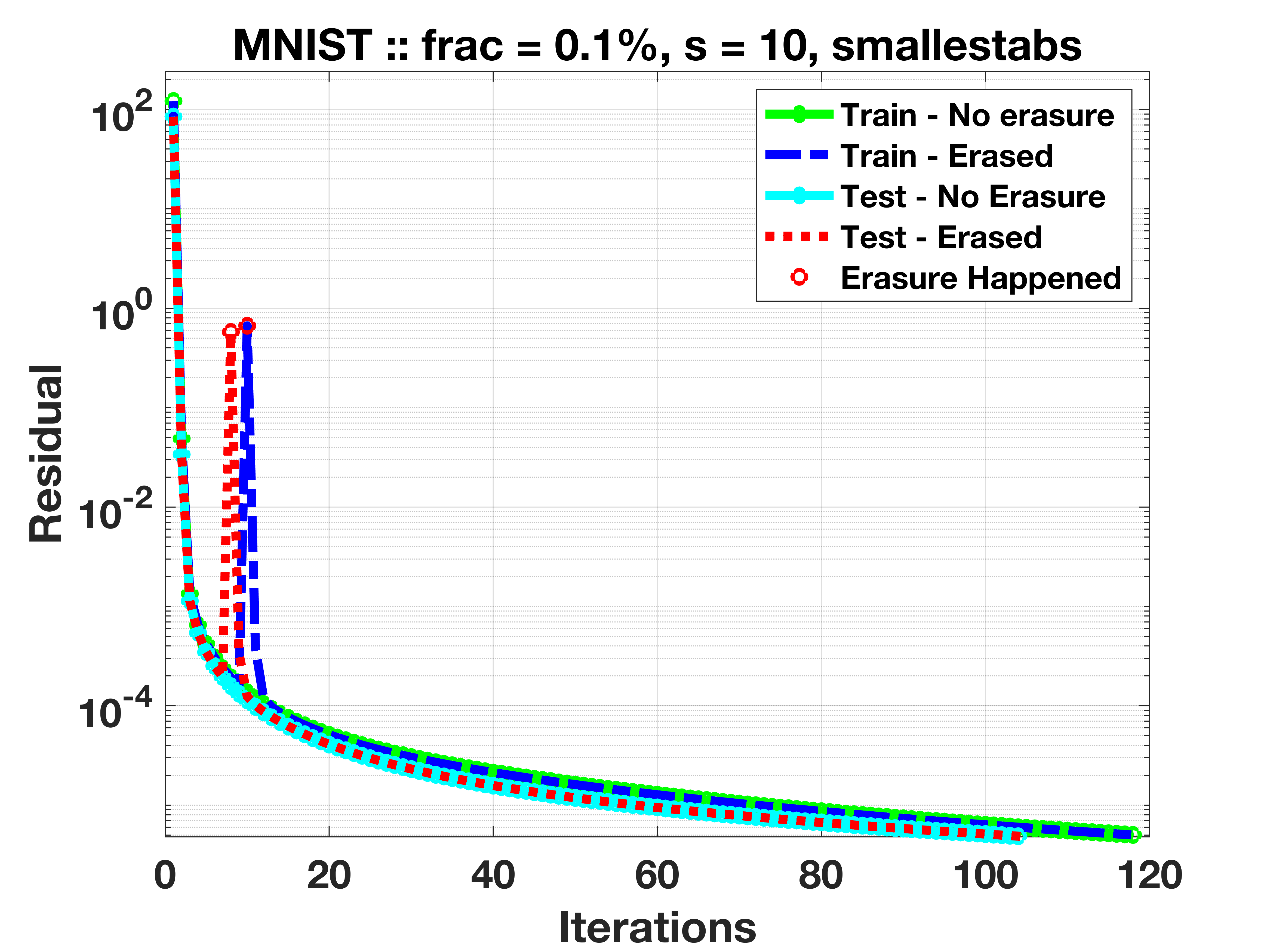}
  \caption{Residual}
  \label{fig:res-iter-mnist-small}
\end{subfigure}%
\begin{subfigure}{.49\textwidth}
  \centering
  \includegraphics[width=.98\linewidth]{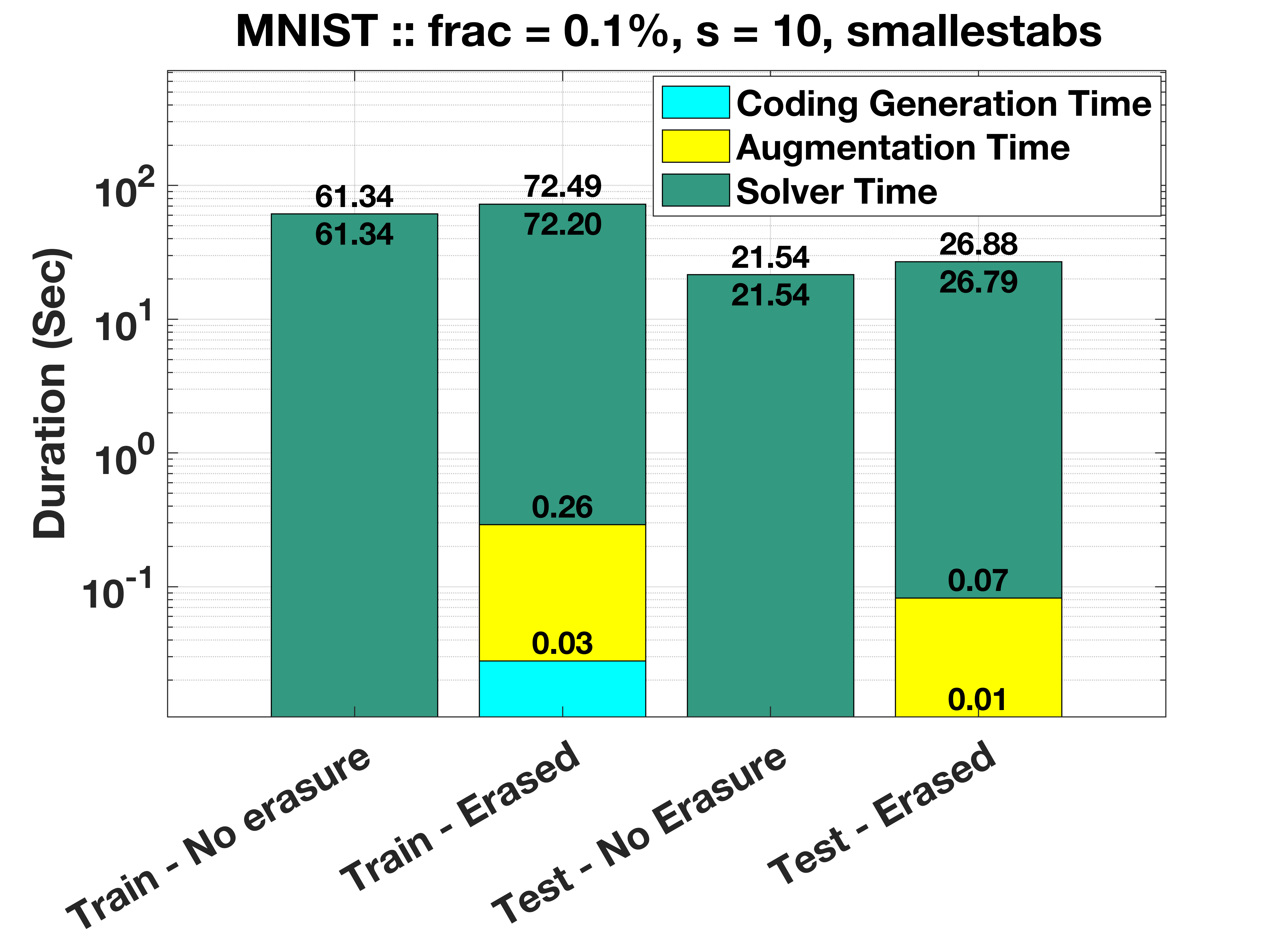}
  \caption{Timing Breakdown}
  \label{fig:res-time-mnist-small}
\end{subfigure}
\caption{Computing Smallest Few Eigenvalue for MNIST Dataset with 1\% Erased}
\label{fig:res-mnist-small}
\end{figure}

Figure \ref{fig:res-mnist-small} shows that at the event of failure, erasure-coded TraceMin converges to its solution without taking many more iterations making the erasure coding an obvious choice. Erasure-Coded TraceMin showing impeccable convergence when computing smallest eigenvalues as well as shown in Figure \ref{fig:res-mnist-small}.

We observe from Figure \ref{fig:res-mnist-small} that erasure-coded TraceMin recovers from 0.1\% erasures (150 row-column pairs) of MNIST Training Dataset to compute exact eigen-pairs (tested up to 15 largest eigen-pairs). The erasure-coded TraceMin takes less than 20\% additional iterations to converge in case of an erasure than the ``No Erasure'' scenario.

\subsection{Multifault}
We implement a Random Multifault model in which multiple faults can occur at any point during execution. These faults are simulated using a random number distribution, where a predefined set of iterations is selected to simulate fault occurrences, resulting in data erasure (loss).

\begin{figure}[H]
\centering
\begin{subfigure}{.49\textwidth}
  \centering
  \includegraphics[width=.98\linewidth]{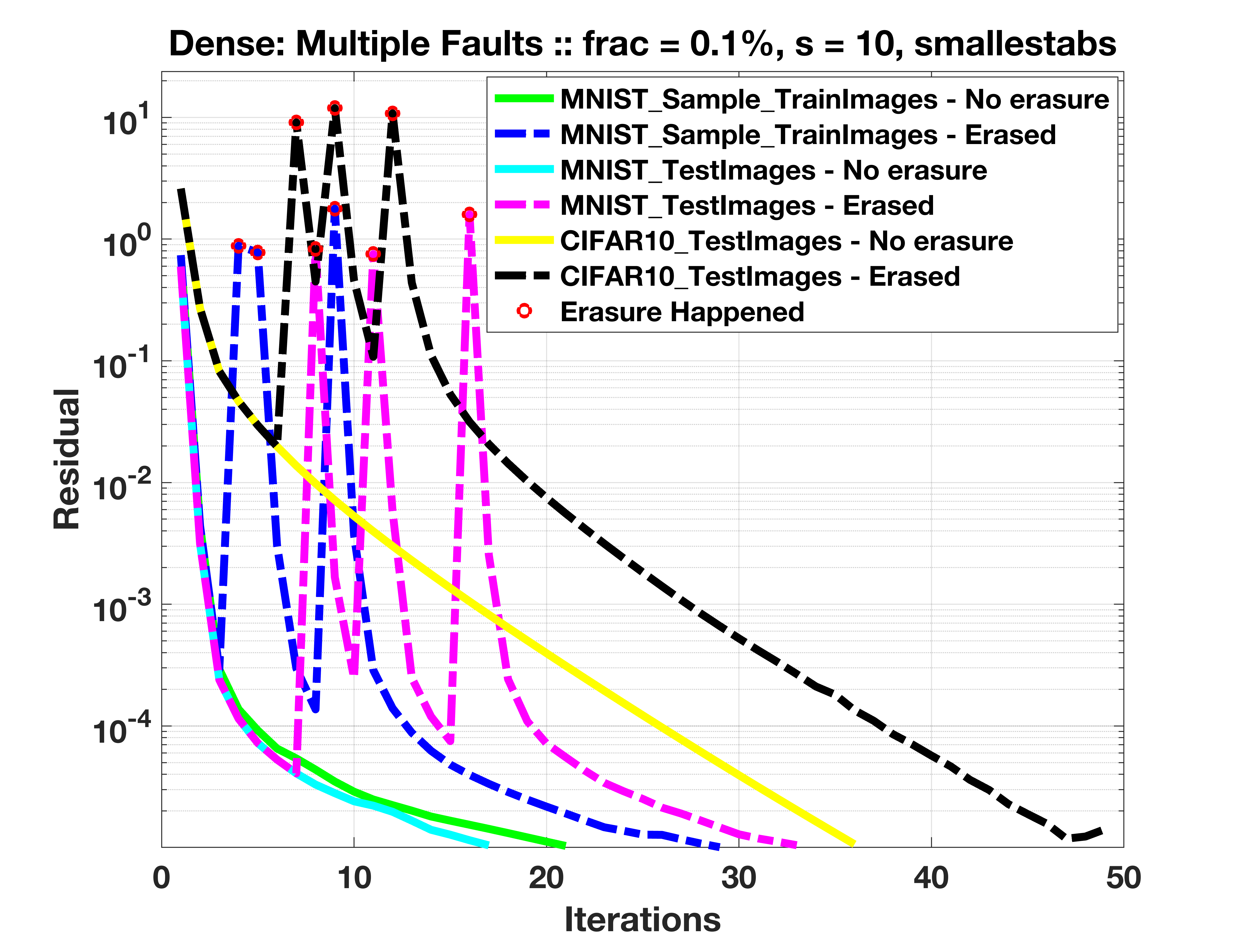}
  \caption{Dense Matrix}
  \label{fig:res-dense-multiple}
\end{subfigure}%
\begin{subfigure}{.49\textwidth}
  \centering
  \includegraphics[width=.98\linewidth]{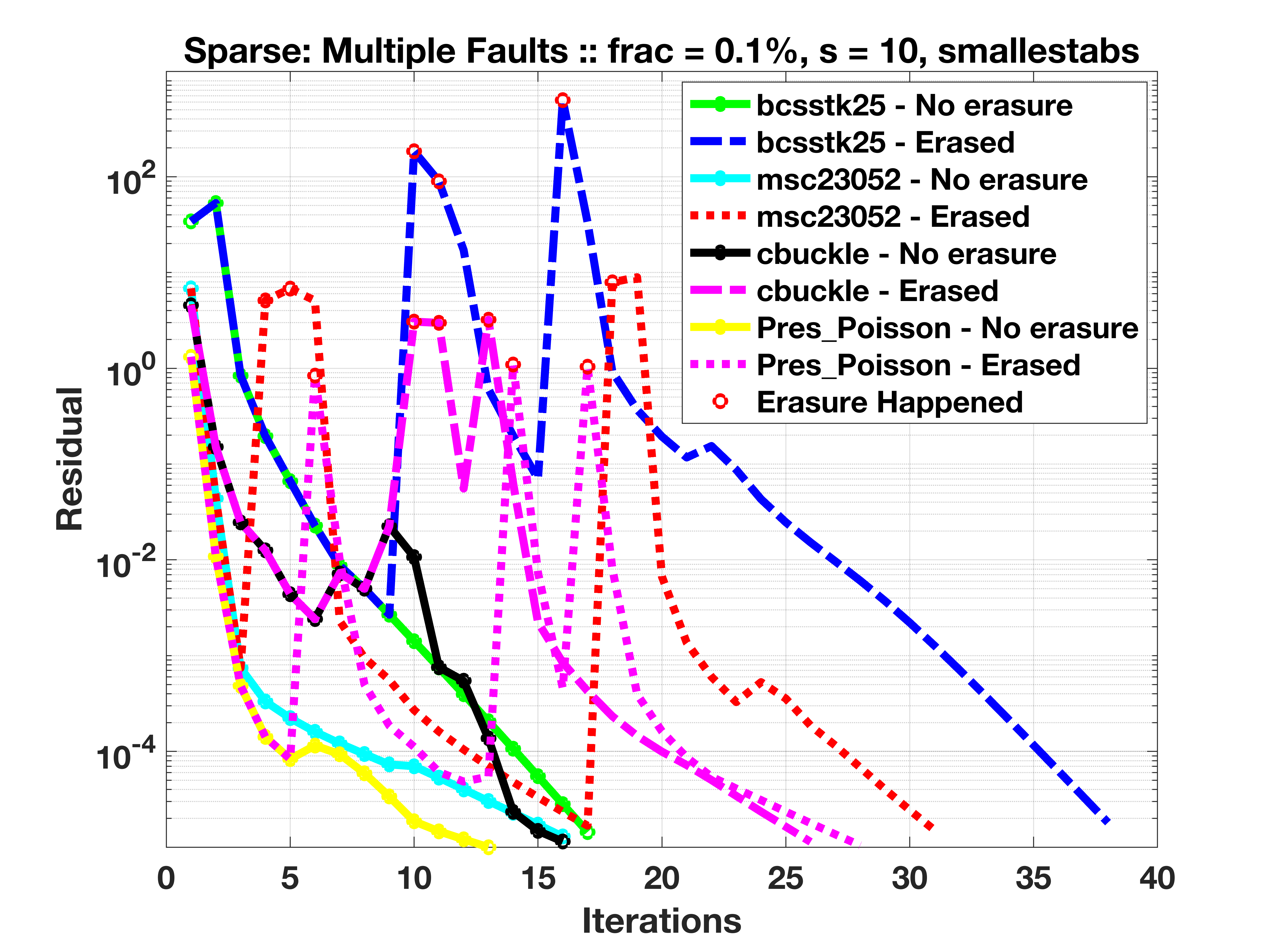}
  \caption{Sparse Matrix}
  \label{fig:res-sparse-multiple}
\end{subfigure}
\caption{Erasure Coded TraceMin in Multifault Scenario}
\label{fig:res-multiple}
\end{figure}

In Figure \ref{fig:res-multiple}, we observe that Erasure-Coded TraceMin successfully computes eigenvalues even in the presence of multiple failures, requiring only a modest increase in iterations. This excellent convergence behavior remains consistent across both sparse and dense datasets. To demonstrate the generalizability of our Erasure-Coded scheme, we select a representative set of dense and sparse datasets. Furthermore, Figure \ref{fig:res-dense-multiple} shows that Erasure-Coded TraceMin requires significantly fewer iterations than the product of the number of faults and the iterations needed by TraceMin in the ``No Erasure'' scenario. In other words, Erasure Coding imposes substantially lower overhead compared to restarting the solver in the event of multiple failures.

\section{Conclusion}
\label{sec:conclusions}

In this work, we introduce a novel erasure-coded, fault-tolerant eigenvalue solver, establishing its correctness and convergence properties. Through extensive evaluations, we demonstrate that our solver incurs minimal overhead across different fault types while maintaining excellent convergence behavior. Our approach exhibits strong robustness and performance across various fault rates and standard benchmark datasets, highlighting its effectiveness in real-world scenarios. By enabling reliable eigenvalue computations in faulty environments, our methods provide a crucial computational foundation for a wide range of scientific applications where resilience and efficiency are paramount.

\section*{Acknowledgments}
We would like to acknowledge NSF for supporting the research by grant.
DFG would like to acknowledge DOE DE-SC0023162 Sparsitute MMICC center for partial support as well as NSF Nonlinear graph IIS-2007481.

\bibliographystyle{siamplain}
\bibliography{references}

\end{document}